\documentclass[reqno,11pt]{amsart}

\usepackage{times,amsmath,cancel,stmaryrd,graphicx}
\usepackage{amsfonts,enumitem,amssymb,color,mathrsfs}
\usepackage[colorlinks=true]{hyperref}
\usepackage{cleveref,enumerate}
\usepackage{lmodern}
\newtheorem{thm}{Theorem}[section]
\newtheorem{lem}[thm]{Lemma}
\newtheorem{cor}[thm]{Corollary}
\newtheorem{prop}[thm]{Proposition}

\newtheorem{rem}[thm]{Remark}

\numberwithin{equation}{section}

\newcommand{\R}{\mathbb{R}}

\newcommand{\rd}{\mathrm{d}}

\newcommand{\dhr}{\mathrel{\lhook\joinrel\relbar\kern-.8ex\joinrel\lhook\joinrel\rightarrow}} 
\newcommand{\bu}{\mathbf{u}}
\newcommand{\bU}{\mathbf{U}}
\newcommand{\bv}{\mathbf{v}}
\newcommand{\bV}{\mathbf{V}}
\definecolor{pansypurple}{rgb}{0.47, 0.09, 0.29}
\definecolor{patriarch}{rgb}{0.5, 0.0, 0.5}
\definecolor{pumpkin}{rgb}{1.0, 0.46, 0.09}
\definecolor{purple(munsell)}{rgb}{0.62, 0.0, 0.77}
\definecolor{spirodiscoball}{rgb}{0.06, 0.75, 0.99}
\definecolor{lightgray}{rgb}{0.83, 0.83, 0.83}
\title{Dynamics of the Reversible Gray-Scott Model and Convergence to its Irreversible Limit}

\author{Philippe Lauren\c{c}ot}
\address{Laboratoire de Math\'ematiques (LAMA) UMR~5127, Universit\'e Savoie Mont Blanc, CNRS\\	F--73000 Chamb\'ery, France}
\email{philippe.laurencot@univ-smb.fr}
\author{Christoph Walker}
\address{Leibniz Universit\"at Hannover\\ Institut f\" ur Angewandte Mathematik \\ Welfengarten 1 \\ D--30167 Hannover\\ Germany}
\email{walker@ifam.uni-hannover.de}

\date{\today}

\begin{document}

\begin{abstract}
Well-posedness of a reversible variant of the Gray-Scott model is shown, along with the convergence of each trajectory to one of the two spatially homogeneous steady states. The principle of linearized stability provides the local attractivity at an exponential rate of the stable steady state, while the long-term limit is identified with the help of center manifold theory. Finally, convergence to the classical Gray-Scott model is proved for an appropriate choice of parameters.
\end{abstract}

\maketitle

\pagestyle{myheadings}
\markboth{\sc{Ph. Lauren\c cot \& Ch. Walker}}{\sc{Reversible Gray-Scott model}}

\section{Introduction}\label{S0}

The Gray-Scott model
\begin{align*}
\partial_t u_1&= d_1\Delta u_1-u_1u_2^2-k_1u_1+a\,, \\
\partial_t u_2&= d_2\Delta u_2+u_1u_2^2 - u_2\,, 
\end{align*}
is a semilinear system of reaction-diffusion equations describing the (irreversible) chemical reactions 
\begin{equation*}
U_1+2U_2 \to 3U_2\,, \quad U_1\to\mathrm{inert}\,, \quad U_2\to\mathrm{inert}\,,
\end{equation*}
with $u_i$ denoting the concentration of the species $U_i$, $i=1,2$. It is by now well known that a wide variety of spatial structures, including spots, stripes and fronts, may be generated by its dynamics depending on the range of the parameters $(d_1,d_2,k_1,a)\in (0,\infty)^4$, see \cite{Cas2017,ChWa2011,DKZ1997,GZK2018,HPT2000,KWW2006,MGR2004,MoKa2004,MuOs2001,NiUe2001,Pea1993,PeWa2009} and the references therein. 

In \cite{LJLWZ2022}, a reversible variant of the Gray-Scott model is derived, featuring four active species involved in the following reversible chemical reactions
\begin{equation*}
    U_1 + 2U_2 \leftrightarrows 3U_2\,, \quad U_1 \leftrightarrows U_4\,, \quad U_2 \leftrightarrows U_3\,,
\end{equation*}
which, after scaling, takes the form
\begin{subequations}\label{rGS}
\begin{align}
\partial_t u_1&= d_1\Delta u_1-u_1u_2^2+k_2u_2^3-k_1u_1+k_4u_4\,, \\
\partial_t u_2&= d_2\Delta u_2+u_1u_2^2-k_2u_2^3-u_2+k_3u_3\,, \\
\partial_t u_3&= d_3\Delta u_3+u_2-k_3u_3\,, \\
\partial_t u_4&= d_4\Delta u_4+k_1u_1-k_4u_4\,.
\end{align}
\end{subequations}
As pointed out in \cite{LJLWZ2022}, the built-in reversibility drastically alters the dynamics: indeed, the system~\eqref{rGS} has an energy (entropy) structure in the sense that
\begin{align*}
    \mathcal{E}_1(u_1,u_2,u_3,u_4) & = \int \left[ u_1 \big(\ln{u_1}-1\big) + u_2 \big(\ln{(k_2 u_2)}-1\big) \right]\,\mathrm{d}x \\
    & \quad + \int \left[ u_3 \big(\ln{(k_2 k_3 u_3)}-1\big) + u_4 \left( \ln{\left( \frac{k_4 u_4}{k_1} \right)} - 1\right) \right]\,\mathrm{d}x
\end{align*}
is a Liapunov functional; that is, it decreases along any trajectory as time increases. As a consequence, no pattern formation arises in the long-term as there are only two stationary solutions, which are both spatially homogeneous. However, the formal computations performed in \cite{HLWY2025,LJLWZ2022} reveal that the classical Gray-Scott model may be recovered as a limit of the reversible Gray-Scott model for a suitable choice of parameters $(k_2,k_3,k_4,d_4)$ and initial value $u_4(0)$. As a consequence, pattern formation is shifted to the transient behavior of solutions to the reversible Gray-Scott model in this particular regime of parameters, as reported in \cite{HLWY2025} on the basis of numerical simulations. Besides, local well-posedness in $H^1$ and global well-posedness for suitably small initial values are established in \cite{LJLWZ2022}, while the local stability of the spatially homogeneous stationary solutions is studied in \cite{HLWY2025}. 

The aim of this paper is to complete and improve the analysis performed in \cite{HLWY2025,LJLWZ2022} in two directions: on the one hand, we prove that the reversible Gray-Scott model is globally well-posed in $L_p$, $p>n$, in a bounded domain of $\mathbb{R}^n$ with homogeneous Neumann boundary conditions, without any smallness condition on the initial data. A key observation here is that the above-mentioned functional $\mathcal{E}_1$ is not the only Liapunov functional available for the reversible Gray-Scott model. Indeed, given any nonnegative convex function $\phi \in W_1^{\infty}((0,\infty))$, one can construct a corresponding Liapunov functional $\mathcal{E}_\phi$, see \Cref{P2}. A consequence of this observation is the convergence of each trajectory to one of the two (spatially homogeneous) stationary steady states. The local stability of these steady states is also studied. On the other hand, we provide a complete proof of the convergence of solutions of the reversible Gray-Scott model to those of the classical Gray-Scott model when $k_2=k_3=k_4=d_4=\varepsilon$ and $u_4(0)=a/\varepsilon$. In contrast to \cite[Theorem~1.1]{LJLWZ2022}, which is only valid on a finite time interval where energy bounds are available, our convergence result is true on any time interval $(0,T)$ and does not require $H^1$-estimates. Instead, it relies on duality techniques to derive $L_{2+\delta}$-estimates on $(u_1,u_2)$ for some $\delta>0$ \cite{CDF2014,DeTr2015}.

\section{Main Results}\label{S1}

Let $(d_1,d_2,d_3,d_4)\in (0,\infty)^4$  and $(k_1,k_2,k_3,k_4)\in (0,\infty)^4$. Consider the reversible Gray-Scott model (after scaling)
\begin{subequations}\label{U}
\begin{align}
\partial_t u_1&= d_1\Delta u_1-u_1u_2^2+k_2u_2^3-k_1u_1+k_4u_4\,, & (t,x)\in (0,\infty)\times\Omega\,,\label{u1}\\
\partial_t u_2&= d_2\Delta u_2+u_1u_2^2-k_2u_2^3-u_2+k_3u_3\,, & (t,x)\in (0,\infty)\times\Omega\,,\label{u2}\\
\partial_t u_3&= d_3\Delta u_3+u_2-k_3u_3\,, & (t,x)\in (0,\infty)\times\Omega\,,\label{u3}\\
\partial_t u_4&= d_4\Delta u_4+k_1u_1-k_4u_4\,, & (t,x)\in (0,\infty)\times\Omega\,,\label{u4}
\end{align}
with a bounded and smooth domain $\Omega\subset\R^n$ subject to homogeneous Neumann boundary conditions
\begin{equation}\label{Ubc)}
    \partial_\nu u_1=\partial_\nu u_2=\partial_\nu u_3=\partial_\nu u_4=0\,,\quad  (t,x)\in (0,\infty)\times\partial\Omega\,,
\end{equation}
and initial conditions
\begin{equation}\label{Uic}
   (u_1,u_2,u_3,u_4)(0)=\big(u_1^0,u_2^0,u_3^0,u_4^0\big)\,,\quad  x\in \Omega\,.
\end{equation}
\end{subequations}

For $p\in [1,\infty]$ we denote the positive cone of $L_p(\Omega)$ by $L_p^+(\Omega)$ and simply write $L_p(\Omega)$ for $L_p(\Omega,\R^4)$ and $L_p^+(\Omega)$ for $\big[L_p^+(\Omega)]^4$. For $\varrho\in\mathbb{R}$, we set 
\begin{equation}
\begin{split}
	\mathcal{Z}_{p,\varrho} & := \left\{ \bu = (u_i)_{1\le i \le 4}\in L_p(\Omega)\ :\ \int_\Omega \sum_{i=1}^4 u_i(x)\,\rd x = \varrho\right\}\,, \\
	\mathcal{Z}_{p,\varrho}^+ & := \mathcal{Z}_{p,\varrho} \cap L_p^+(\Omega)\,.
\end{split} \label{fz}
\end{equation}

The global well-posedness of~\eqref{U} may be formulated as follows.

\begin{thm}\label{T1}
Let $p\in (n,\infty)$ and $\varrho>0$. Given $\bu^0\in \mathcal{Z}_{p,\varrho}$ there exists a unique global strong solution $\bu=\bu(\cdot;\bu^0)$ to~\eqref{U} in $\mathcal{Z}_{p,\varrho}$; that is,
\begin{equation}
\begin{split}
	\bu & \in C\big([0,\infty),\mathcal{Z}_{p,\varrho}\big)\,, \\
	\bu & \in C^1\big((0,\infty),L_p(\Omega)\big)\cap  C\big((0,\infty), W_{p}^2(\Omega)\big)\,.
\end{split} \label{reg}
\end{equation}
Moreover, for any $t_0>0$, 
\begin{equation}
	\sup_{t\ge t_0} \|\bu(t)\|_{W_p^1} < \infty\ , \label{unifb}
\end{equation}
and, if $\bu^0\in \mathcal{Z}_{p,\varrho}^+$, then $\bu(t)\in \mathcal{Z}_{p,\varrho}^+$ for all $t\ge 0$. 

In fact, the mapping $(t,\bu^0)\mapsto \bu(t;\bu^0)$ defines a global semiflow on $\mathcal{Z}_{p,\varrho}$ and each orbit $\{\bu(t;\bu^0)\ :\ t\ge 0\}$ is relatively compact in $\mathcal{Z}_{p,\varrho}$.
\end{thm}

Though only non-negative solutions to~\eqref{U} are physically relevant, the well-posedness for arbitrary initial conditions is needed later on for the analysis of the stability of the steady state lying on the boundary of $\mathcal{Z}_{p,\varrho}^+$ (see~\eqref{eqr} below).

We next turn to the long-term behavior of non-negative solutions to~\eqref{U} and fix $p\in (n,\infty)$ and $\varrho>0$. We first recall that, according to \cite[Section~1]{LJLWZ2022}, there are only two stationary solutions to~\eqref{U} in $\mathcal{Z}_{p,\varrho}^+$, denoted by $\mathbf{E}_{b,\varrho}$ and $\mathbf{E}_{\circ,\varrho}$, which are both spatially homogeneous and given by
\begin{equation}
\begin{split}
	\mathbf{E}_{\circ,\varrho} := \frac{\varrho}{|\Omega|} \mathbf{E}_{\circ}\,,  & \qquad  \mathbf{E}_{\circ} := \frac{1}{K_0} \big(k_2k_3k_4\,, k_3k_4\,, k_4\,, k_1k_2k_3\big)\,, \\
	\mathbf{E}_{b,\varrho} := \frac{\varrho}{|\Omega|} \mathbf{E}_{b}\,,  & \qquad \mathbf{E}_{b} := \frac{1}{k_1+k_4} \big( k_4\,, 0\,, 0\,, k_1 \big)\,,
\end{split} \label{eqr}
\end{equation}  
with
\begin{equation*}
	K_0 := k_1k_2k_3 + k_2k_3k_4 + k_3k_4 + k_4\,.
\end{equation*}
The set of equilibria in the invariant subset $\mathcal{Z}_{p,\varrho}^+$ of $L_p^+(\Omega)$ is thus discrete, a property which we combine with the availability of Liapunov functionals, see \Cref{S2}, and the semiflow and compactness properties provided in \Cref{T1} to identify the long-term behavior of solutions to~\eqref{U}.

\begin{thm}\label{T2}
Let $p\in (n,\infty)$, $p\ge 3$, and $\varrho>0$. Given $\bu^0\in \mathcal{Z}_{p,\varrho}^+$, there is $*\in\{b,\circ\}$ (depending on $\bu^0$) such that
\begin{equation}
	\lim_{t\to\infty} \|\bu(t) - \mathbf{E}_{*,\varrho} \|_p=0\,. \label{ltb}
\end{equation}	
In fact,
\begin{itemize}[label = $\triangleright$]
    \item $\mathbf{E}_{*,\varrho}=\mathbf{E}_{b,\varrho}$ in~\eqref{ltb} if and only if $u_2^0=u_3^0=0$;
    \item $\mathbf{E}_{\circ,\varrho}$ is locally exponentially stable in $\mathcal{Z}_{p,\varrho}^+$.
\end{itemize}
\end{thm}

The proof of \Cref{T2} relies on the study of the linearization of~\eqref{U} at $\mathbf{E}_{*,\varrho}$ with $*\in\{\circ,b\}$: for $\mathbf{E}_{\circ,\varrho}$, we prove that the spectrum of the linearized operator in $\mathcal{Z}_{p,0}$ is contained in $(-\infty,-\lambda_{\circ,\varrho}]$ for some $\lambda_{\circ,\varrho}>0$ and the asymptotic exponential stability of $E_{\circ,\varrho}$ is then derived with the help of the principle of linearized stability. This efficient tool is no longer available for $\mathbf{E}_{b,\varrho}$ as the spectrum of the linearized operator in $\mathcal{Z}_{p,0}$ contains the zero eigenvalue. Thus, $\mathbf{E}_{b,\varrho}$ is only semi-stable and it is unclear at first glance whether it plays a role in the long-term dynamics. To shed some light on this issue, we appeal to center manifold theory \cite{Carr81} and construct a one-dimensional center manifold on which the dynamics is governed by an ordinary differential equation of the form $X'=-A X^2$ for some $A>0$, from which a one-side stability (for $X(0)>0$) of $\mathbf{E}_{b,\varrho}$ follows, which however excludes non-negative initial values  with non-zero second and third components. Both results complete and extend the analysis performed in \cite[Sections~3.1-3.2]{HLWY2025}, which is based on minimizing properties of the steady states and numerical simulations.

\begin{rem}\label{R1}
If $\bu^0=(u_i^0)_{1\le i\le 4}\in \mathcal{Z}_{p,\varrho}^+$ is such that $u_2^0=u_3^0=0$, then the solution $\bu=(u_i)_{1\le i \le 4}$ to~\eqref{U} satisfies 
$u_2(t)=u_3(t)=0$ for all $t\ge 0$, while $(u_1,u_4)$ solves the reduced linear system
\begin{subequations}\label{rU}
\begin{align}
\partial_t u_1&= d_1\Delta u_1-k_1u_1+k_4u_4\,, & (t,x)\in (0,\infty)\times\Omega\,,\label{ru1}\\
\partial_t u_4&= d_4\Delta u_4+k_1u_1-k_4u_4\,, & (t,x)\in (0,\infty)\times\Omega\,,\label{ru4}
\end{align}
\end{subequations}
subject to homogeneous Neumann boundary conditions and initial conditions $(u_1^0,u_4^0)$. The convergence of $\bu(t)$ to $\mathbf{E}_{b,\varrho}$ is obvious in that case, as $k_1\|u_1-E_{b,\varrho,1}\|_2^2+k_4\|u_4-E_{b,\varrho,4}\|^2$ is a Liapunov functional for~\eqref{rU}.
\end{rem}

The last contribution of this paper is a proof of the convergence of solutions of the reversible Gray-Scott model~\eqref{U} to that of the classical Gray-Scott model when the reaction rates and diffusion coefficients are appropriately chosen, as outlined in \cite[Section~1.1]{LJLWZ2022}. Let us first recall that the classical Gray-Scott model reads
\begin{subequations}\label{GS}
\begin{align}
\partial_t u_1&= d_1\Delta u_1-u_1u_2^2-k_1u_1+a\,, & (t,x)\in (0,\infty)\times\Omega\,,\label{gs1}\\
\partial_t u_2&= d_2\Delta u_2+u_1u_2^2 - u_2\,, & (t,x)\in (0,\infty)\times\Omega\,,\label{gs2}
\end{align}
subject to homogeneous Neumann boundary conditions
\begin{equation}\label{GSbc)}
    \partial_\nu u_1=\partial_\nu u_2=0\,,\quad  (t,x)\in (0,\infty)\times\partial\Omega\,,
\end{equation}
and initial conditions
\begin{equation}\label{GSic}
   (u_1,u_2)(0)=\big(u_1^0,u_2^0\big)\,,\quad  x\in \Omega\,,
\end{equation}
\end{subequations}
where $a$ is a non-negative source term (usually taken to be a positive constant in the literature). The global well-posedness in a classical sense is established in \cite{HMP1987}.

\begin{thm}\label{T3}
Consider $\varepsilon\in (0,1)$ and assume that
\begin{equation}
    k_2=k_3=k_4=d_4=\varepsilon \,. \label{il01}
\end{equation}
Given $(u_1^0,u_2^0,u_3^0,a)\in W_{p_0}^{1,+}(\Omega)$ for some $p_0\in (n,\infty)$, $p_0\ge 2$, we set $\bu_\varepsilon^0= (u_1^0,u_2^0,u_3^0,a/\varepsilon)$ and let $\bu_\varepsilon=\bu(\cdot;\bu_\varepsilon^0)$ be the corresponding solution to~\eqref{U}. Then there is $p\in (1,2)$ such that, for any $T>0$,
\begin{equation*}
       \lim_{\varepsilon\to 0} \sup_{t\in [0,T]} \|(u_{i,\varepsilon}-u_i)(t)\|_{p} = 0\,, \quad 1\le i \le 2\,, 
\end{equation*}
where $(u_1,u_2)$ is the classical solution to~\eqref{GS}. In addition,
\begin{equation*}
    \lim_{\varepsilon\to 0} \sup_{t\in [0,T]} \|(u_{3,\varepsilon}-u_3)(t)\|_{p} = \lim_{\varepsilon\to 0} \sup_{t\in [0,T]} \|\varepsilon u_{4,\varepsilon}(t) - a\|_{2} = 0\,,
\end{equation*}
where $u_3$ is the unique classical solution to
\begin{subequations}\label{heq3}
\begin{align}
    \partial_t u_3 & = d_3 \Delta u_3 + u_2 \,, & (t,x)\in (0,\infty)\times\Omega\,, \label{heq3e}\\
    \partial_\nu u_3 & = 0\,, &  (t,x)\in (0,\infty)\times\partial\Omega\,,\label{heq3bc}\\
    u_3(0) & = u_3^0\,, & x\in\Omega\,. \label{heq3ic}
\end{align}
\end{subequations}
\end{thm}

The outline of this paper is as follows. In the following Section~\ref{S2} we derive Liapunov functionals and establish Theorem~\ref{T1} on the global well-posedness of the reversible Gray-Scott model~\eqref{U}. In Section~\ref{S3} we address the long-term behavior of solutions as summarized in Theorem~\ref{T2}. Finally, in Section~\ref{S4} we prove the convergence of solutions to the reversible Gray-Scott model~\eqref{U} to solutions of the irreversible Gray-Scott model~\eqref{GS} stated in Theorem~\ref{T3}.

\section{Well-Posedness and Liapunov Functionals}\label{S2}

From now on, we denote the Laplace operator in $L_p(\Omega)$ with domain
\begin{equation*}
    W_{p,N}^2(\Omega) := \left\{ z\in W_p^2(\Omega)\ :\ \partial_\nu z = 0 \;\text{ on }\; \partial\Omega \right\}
\end{equation*}
by $\Delta_N$. We first establish the local well-posedness of \eqref{U}.

\begin{prop}\label{P1}
Let $p\in (n,\infty)$ and $\varrho>0$. Given $\bu^0\in \mathcal{Z}_{p,\varrho}$ there exists a unique maximal strong solution $\bu=\bu(\cdot;\bu^0)$ to~\eqref{U} in $\mathcal{Z}_{p,\varrho}$ defined on a maximal time interval $[0,t^+(\bu^0))$ with $t^+(\bu^0)\in (0,\infty]$; that is,
\begin{equation}
\begin{split}
	\bu & \in C\big([0,t^+(\bu^0)),\mathcal{Z}_{p,\varrho}\big)\,, \\
	\bu & \in C^1\big((0,t^+(\bu^0)),L_p(\Omega)\big)\cap  C\big((0,t^+(\bu^0)),W_{p}^2(\Omega)\big)\,.
\end{split}\label{regl}
\end{equation}
In addition, if $\bu^0\in \mathcal{Z}_{p,\varrho}^+$, then $\bu(t)\in\mathcal{Z}_{p,\varrho}^+$ for $t\in [0,t^+(\bu^0))$.

If $t^+(\bu^0)<\infty$, then
\begin{equation}\label{bcrit}
	\limsup_{t\nearrow t^+(\bu^0)}\|\bu(t)\|_{p}=\infty\,.
\end{equation}
In fact, the mapping $(t,\bu^0)\mapsto \bu(t;\bu^0)$ defines a semiflow on $\mathcal{Z}_{p,\varrho}$. 
\end{prop}

\begin{proof}
Introducing
\begin{equation}\label{AA}
A:=\mathrm{diag}[d_1\Delta_N-k_1,d_2\Delta_N-1,d_3\Delta_N-k_3,d_4\Delta_N-k_4]
\end{equation}
and gathering the remaining linear and the nonlinear terms, it is readily seen that we can write~\eqref{U} as an abstract semilinear Cauchy problem
\begin{equation}\label{CP}
	\frac{\rd \bu}{\rd t}=A\bu+f(\bu)\,,\quad t>0\,,\qquad \bu(0)=\bu^0\,.
\end{equation}
Hereby, $A$ is the generator of a positive, strongly continuous analytic semigroup on the phase space $E_0:=L_p(\Omega)$ with domain $E_1:=W_{p,N}^2(\Omega)$, and the nonlinearity $f$ is given by 
\begin{equation}
    f(\bu) := \big( -u_1 u_2^2 + k_2 u_2^3 + k_4 u_4, u_1u_2^2 - k_2 u_2^3 + k_3 u_3, u_2, k_1 u_1\big)\label{deff}
\end{equation}
and satisfies
\begin{equation*}
	\|f(\bu)-f(\bv)\|_{p}\le  c\big( 1 + \|\bu\|_{\infty}^2+\|\bv\|_{\infty}^2\big)\|\bu-\bv\|_{p}\,,\quad (\bu,\bv)\in L_\infty(\Omega)\times L_\infty(\Omega)\,. 
\end{equation*}
Since $p>n$ we can fix $\xi\in \big(n/2p,1/2\big)$ so that 
\begin{equation}\label{xi}
	E_\xi:=(E_0,E_1)_{\xi,p}= W_{p}^{2\xi}(\Omega)\hookrightarrow L_\infty(\Omega)\,,
\end{equation}
see, e.g., \cite[4.3.3/Theorem]{Triebel}. Therefore, 
\begin{equation}\label{f}
	\|f(\bu)-f(\bv)\|_{E_0}\le c \big( 1 + \|\bu\|_{E_\xi}^2+\|\bv\|_{E_\xi}^2\big) \|\bu-\bv\|_{E_0}\,, \qquad (\bu,\bv)\in E_\xi^2\,.
\end{equation}
We may thus apply \cite[Theorem~1.2]{MW_PRSE} (with $\alpha=\gamma=0<\xi<1/2$ and $q=2$ therein) to deduce that~\eqref{CP} has for each $\bu^0\in E_0=L_p(\Omega)$ a unique maximal strong solution $\bu(\cdot;\bu^0)$ with regularity~\eqref{regl} and satisfying the blow-up criterion~\eqref{bcrit}. Moreover, the mapping $(t,\bu^0)\mapsto \bu(t;\bu^0)$ defines a semiflow on $L_p(\Omega)$. 

If $\bu^0\in L_p^+(\Omega)$, then $\bu(t;\bu^0)\in L_p^+(\Omega)$ for $t\in [0,t^+(\bu^0))$ since the semigroup generated by $A$ is positive and since for each $R>0$ there is $c(R)>0$ such that
\begin{equation*}
	f(\bv)+ c(R)\bv\ge 0\,,\qquad \bv\in L_\infty^+(\Omega) \ \text{ with }\ \|\bv\|_\infty\le R\,.
\end{equation*}
Finally, the property $\bu(t)\in \mathcal{Z}_{p,\varrho}$ for $t\in [0,t^+(\bu^0))$ is a consequence of the Neumann boundary conditions~\eqref{Ubc)} and the fact that the reaction terms on the right-hand side of~\eqref{u1}-\eqref{u4} sum up to zero.
\end{proof}

Next, we provide a broad class of Liapunov functionals for~\eqref{U} that we shall use to derive global existence of solutions and to investigate their long-term behavior.

\begin{prop}\label{P2}
Let $p\in (n,\infty)$ and $\varrho>0$. Given $\bu^0\in \mathcal{Z}_{p,\varrho}$ and a non-negative convex function $\phi\in C^2(\mathbb{R})$ such that
\begin{equation}
	\mathcal{E}_\phi(\bu^0) := \int_\Omega \left[ \phi(u_1^0) + \frac{1}{k_2} \phi(k_2 u_2^0) +  \frac{1}{k_2 k_3} \phi(k_2 k_3 u_3^0) + \frac{k_1}{k_4} \phi\left( \frac{k_4}{k_1} u_4^0\right) \right]\,\rd x   < \infty\,, \label{lfphi0}
\end{equation}
the corresponding strong solution $\bu = \bu(\cdot;\bu^0)$ to~\eqref{U} provided by \Cref{P1} satisfies $\mathcal{E}_\phi(\bu(t))\le \mathcal{E}_\phi(\bu^0)$ for all $t\in [0,t^+(\bu^0))$. More precisely,
\begin{equation}
	\frac{\rd}{\rd t} \mathcal{E}_\phi(\bu) + \mathcal{D}_\phi[\bu] + \mathcal{R}_\phi[\bu] = 0\,, \qquad t\in (0,t^+(\bu^0))\,, \label{lfphi}
\end{equation}
with $\mathcal{D}_\phi[\bu]\ge 0$ and $\mathcal{R}_\phi[\bu]\ge 0$ given by
\begin{align*}
	\mathcal{D}_\phi[\bu] & := \int_\Omega \left[ d_1 \phi''(u_1) |\nabla u_1|^2 + d_2 k_2 \phi''(k_2 u_2) |\nabla u_2|^2 \right]\,\rd x \\
	& \quad + \int_\Omega \left[d_3 k_2 k_3 \phi''(k_2k_3u_3) |\nabla u_3|^2 + \frac{d_4 k_4}{k_1} \phi''\left(\frac{k_4}{k_1} u_4\right) |\nabla u_4|^2\right]\,\rd x
\end{align*}
and
\begin{align*}
	\mathcal{R}_\phi[\bu] & := \int_\Omega u_2^2 (u_1-k_2 u_2) \big( \phi'(u_1) - \phi'(k_2 u_2) \big)\,\rd x  \\
	& \quad + \frac{1}{k_2} \int_\Omega (k_2u_2 - k_2 k_3 u_3) \big(\phi'(k_2 u_2) - \phi'(k_2 k_3 u_3) \big)\,\rd x \\ 
	& \quad + k_1 \int_\Omega \left( u_1 - \frac{k_4}{k_1} u_4 \right) \left( \phi'(u_1) - \phi'\left( \frac{k_4}{k_1} u_4 \right) \right)\,\rd x \,.
\end{align*}
\end{prop}

Owing to the non-negativity of $\mathcal{D}_\phi[\bu]$ and $\mathcal{R}_\phi[\bu]$, the identity~\eqref{lfphi} guarantees that~$\mathcal{E}_\phi$ is a Liapunov functional for~\eqref{U}, so that $t\mapsto \mathcal{E}_\phi[\bu(t)]$ is a non-increasing function on~$[0,t^+(\bu^0))$.

When $\phi(r)=r\ln{r}-r+1$, $r\ge 0$, \Cref{P2} is reported in \cite[Eq.~(1.5)]{LJLWZ2022} for non-negative solutions. We extend it here to an arbitrary convex function.

\begin{proof}
We readily infer from~\eqref{U} that
\begin{align*}
	\frac{\rd}{\rd t} \int_\Omega \phi(u_1)\,\rd x & = - d_1 \int_\Omega \phi''(u_1) |\nabla u_1|^2\,\rd x - \int_\Omega u_2^2 (u_1-k_2u_2) \phi'(u_1)\,\rd x \\
	& \quad - k_1 \int_\Omega  \left( u_1 - \frac{k_4}{k_1} u_4 \right) \phi'(u_1)\,\rd x\,, 
\end{align*}
\begin{align*}
	\frac{1}{k_2} \frac{\rd}{\rd t} \int_\Omega \phi(k_2 u_2)\,\rd x & = - d_2 k_2 \int_\Omega \phi''(k_2 u_2) |\nabla u_2|^2\,\rd x + \int_\Omega u_2^2 (u_1-k_2u_2) \phi'(k_2u_2)\,\rd x \\
	& \quad - \frac{1}{k_2} \int_\Omega  \left( k_2u_2 - k_2k_3 u_3 \right) \phi'(k_2u_2)\,\rd x\,, 
\end{align*}
\begin{align*}
	\frac{1}{k_2k_3} \frac{\rd}{\rd t} \int_\Omega \phi(k_2k_3u_3)\,\rd x & = - d_3k_2k_3 \int_\Omega \phi''(k_2k_3u_3) |\nabla u_3|^2\,\rd x \\
	& \quad + \frac{1}{k_2} \int_\Omega (k_2u_2-k_2k_3u_3) \phi'(k_2k_3u_3)\,\rd x \,, 
\end{align*}
and
\begin{align*}
	\frac{k_1}{k_4} \frac{\rd}{\rd t} \int_\Omega \phi\left( \frac{k_4}{k_1} u_4 \right)\,\rd x & = - \frac{d_4 k_4}{k_1} \int_\Omega \phi''\left( \frac{k_4}{k_1} u_4 \right) |\nabla u_4|^2\,\rd x \\
	& \quad + k_1 \int_\Omega  \left( u_1 - \frac{k_4}{k_1} u_4 \right) \phi'\left( \frac{k_4}{k_1} u_4 \right)\,\rd x\,.
\end{align*}
Summing up the above four identities leads us to~\eqref{lfphi}. The non-negativity of $\mathcal{D}_\phi[\bu]$ and $\mathcal{R}_\phi[\bu]$ follows directly from the convexity of $\phi$. Combined with~\eqref{lfphi}, this property guarantees that $\mathcal{E}_\phi$ is a Liapunov functional for~\eqref{U}.
\end{proof}

To simplify the notation, we set $\mathcal{E}_p := \mathcal{E}_{\phi_p}$ for $p>1$ when $\phi_p(r) := |r|^p$ for $r\in\mathbb{R}$.

We now derive several consequences of \Cref{P2}, beginning with the global well-posedness of~\eqref{U} and the uniform boundedness of its solutions.

\begin{proof}[Proof of \Cref{T1}]
Let $p\in (n,\infty)$. Given $\varrho>0$ and $\bu^0\in \mathcal{Z}_{p,\varrho}$, the condition \mbox{$\mathcal{E}_p(\bu^0)<\infty$} is obviously satisfied and we readily deduce from \Cref{P2} that the corresponding solution $\bu=\bu(\cdot;\bu^0)$ to~\eqref{U} satisfies
\begin{equation*}
	\mathcal{E}_p(\bu(t))\le \mathcal{E}_p(\bu^0)\,, \qquad t\in [0,t^+(\bu^0))\,,
\end{equation*}
which clearly excludes the occurrence of~\eqref{bcrit}. Consequently, $t^+(\bu^0)=\infty$ and the global well-posedness of~\eqref{U}, along with the global semiflow property, is established. 

Let us next fix $t_0>0$. It follows from the continuous embedding of $W_{p,N}^2(\Omega)$ in $L_\infty(\Omega)$ that
\begin{equation}
	M(t_0) := \max\left\{ \|u_1(t_0)\|_\infty\,, k_2 \|u_2(t_0)\|_\infty\,, k_2k_3 \|u_3(t_0)\|_\infty\,, \frac{k_4}{k_1} \|u_4(t_0)\|_\infty \right\} <\infty\,. \label{supt0}
\end{equation}
Introducing the non-negative convex function 
\begin{equation*}
\Phi_{M(t_0)}(r) := \big( r - M(t_0) \big)_+ = \max\big\{ r-M(t_0)\,, 0 \big\}\,,\quad r\in\mathbb{R}\,,
\end{equation*}
we deduce from~\eqref{supt0} that $\mathcal{E}_{\Phi_{M(t_0)}}(\bu^0)=0$, while \Cref{P2}, along with an approximation argument to cope with the non-differentiability of the positive part, implies that
\begin{equation*}
	0 \le \mathcal{E}_{\Phi_{M(t_0)}}(\bu(t)) \le \mathcal{E}_{\Phi_{M(t_0)}}(\bu^0) = 0\,, \qquad t\ge t_0\,.
\end{equation*}
Combining the above property with the definition of $\mathcal{E}_{\Phi_{M(t_0)}}$, we find
\begin{equation*}
	\max\left\{ u_1(t,x)\,, k_2 u_2(t,x)\,, k_2k_3 u_3(t,x)\,, \frac{k_4}{k_1} u_4(t,x) \right\} \le M(t_0)
\end{equation*}
for $(t,x)\in [t_0,\infty)\times\Omega$. A similar argument involving the non-negative convex function $\Psi_{M(t_0)}$ defined by $\Psi_{M(t_0)}(r) = \big( -r-M(t_0) \big)_+$ for $r\in\mathbb{R}$ shows that
\begin{equation*}
	\min\left\{ u_1(t,x)\,, k_2 u_2(t,x)\,, k_2k_3 u_3(t,x)\,, \frac{k_4}{k_1} u_4(t,x) \right\} \ge - M(t_0)
\end{equation*}
for $(t,x)\in [t_0,\infty)\times\Omega$. Consequently,
\begin{equation}
	\max\left\{ \|u_1(t)\|_\infty\,, k_2 \|u_2(t)\|_\infty\,, k_2k_3 \|u_3(t)\|_\infty\,, \frac{k_4}{k_1} \|u_4(t)\|_\infty \right\} \le M(t_0)\,, \qquad t\ge t_0\,. \label{ulb}
\end{equation}
In particular, $f(\bu)$ belongs to $L_\infty((t_0,\infty)\times\Omega)$. We then infer from~\eqref{CP}, the exponential decay of the semigroup generated by $A$ defined in~\eqref{AA} (due to the positivity of $(k_1,k_3,k_4)$), and parabolic regularity estimates that $\bu$ belongs to $L_\infty((t_0,\infty),W_p^1(\Omega))$, thereby establishing~\eqref{unifb}. Due to the compact embedding of $W_p^1(\Omega)$ in $L_p(\Omega)$, the stated relative compactness of the orbits in $L_p(\Omega)$ readily follows. This completes the proof of \Cref{T1}.
\end{proof}

\section{Long-Term Behavior}\label{S3}

Having established the global well-posedness of~\eqref{U}, as well as the boundedness of its solutions, we now turn to their qualitative behavior. We first exploit the availability of Liapunov functionals and the semiflow property to show the long-term convergence of each solution to a steady state, see \Cref{S3.1}. The second step is to identify the limiting steady state, which requires a more detailed study of the semiflow in a neighborhood of each steady state, see \Cref{S3.2} for $\mathbf{E}_{\circ,\varrho}$ and \Cref{S3.3} for $\mathbf{E}_{b,\varrho}$. 

\subsection{Long-Term Convergence}\label{S3.1}

We first establish that $\mathcal{E}_2$ is a strict Liapunov functional for~\eqref{U} in $\mathcal{Z}_{p,\varrho}^+$.

\begin{lem}\label{L1}
Let $p\in (n,\infty)$, $p\ge 2$, and $\varrho>0$. Let $\bu^0\in \mathcal{Z}_{p,\varrho}^+$ be such that the corresponding solution $\bu=\bu(\cdot;\bu^0)$ satisfies
\begin{equation}
	\mathcal{E}_2(\bu(t)) = \mathcal{E}_2(\bu^0)\,, \qquad t\ge 0\,. \label{ls01}
\end{equation}	
Then there is $*\in\{b,\circ\}$ such that $\bu(t)=\bu^0=\mathbf{E}_{*,\varrho}$ for all $t\ge 0$.
\end{lem}

\begin{proof}
It readily follows from~\eqref{lfphi} and~\eqref{ls01} that
\begin{equation}
	\mathcal{D}_2[\bu(t)] = \mathcal{R}_2[\bu(t)] = 0\,, \qquad t\ge 0\,, \label{ls02}
\end{equation}
which implies, thanks to the positivity of the diffusion and reaction coefficients, that
\begin{subequations}\label{ls03}
\begin{equation}
	u_i(t,x) = U_i(t) := \frac{1}{|\Omega|} \int_\Omega u_i(t,y)\,\rd y\,, \quad (t,x)\in  [0,\infty)\times\Omega\,, \ 1\le i \le 4\,, \label{ls03a}
\end{equation}
and
\begin{equation}
	U_1(t) = \frac{k_4}{k_1} U_4(t)\,, \ U_3(t) = \frac{1}{k_3} U_2(t)\,, \ U_2^2(t) (U_1(t)-k_2 U_2(t)) = 0\,, \quad t\ge 0\,. \label{ls03b}
\end{equation}
\end{subequations}
Combining~\eqref{ls01}, \eqref{ls03} and the property $\bu(t)\in\mathcal{Z}_{p,\varrho}^+$ provides two additional identities relating the components of $\bU(t):=(U_i(t))_{1\le i\le 4}$, namely,
\begin{equation}
	\left( 1 + \frac{k_1}{k_4} \right) U_1^2(t) + k_2 \left( 1 + \frac{1}{k_3} \right) U_2^2(t) = \frac{\mathcal{E}_2(\bu^0)}{|\Omega|} \,, \qquad t\ge 0\,, \label{ls04} 
\end{equation}
and
\begin{equation}
	\left( 1 + \frac{k_1}{k_4} \right) U_1(t) + \left( 1 + \frac{1}{k_3} \right) U_2(t) = \frac{\varrho}{|\Omega|}\,, \qquad t\ge 0\,. \label{ls05} 
\end{equation}
Plugging~\eqref{ls05} into~\eqref{ls04} leads us to a single algebraic equation for $U_1$:
\begin{equation*}
	\frac{k_1+k_4}{k_4} U_1^2(t) + \frac{k_2k_3}{1+k_3} \left( \frac{\varrho}{|\Omega|} - \frac{k_1+k_4}{k_4} U_1(t) \right)^2 = \frac{\mathcal{E}_2(\bu^0)}{|\Omega|}\,, \qquad t\ge 0\,.
\end{equation*}
Owing to the regularity of $\bu$, see \Cref{T1}, we may differentiate the above identity with respect to time and find
\begin{equation*}
	\left[ \left( 1 + \frac{k_2k_2(k_1+k_4)}{k_4(1+k_3)} \right) U_1(t) - \frac{k_2 k_3}{1+k_3} \frac{\varrho}{|\Omega|} \right] U_1'(t) = 0\,, \qquad t\ge 0\,,
\end{equation*}
from which we deduce that $U_1(t) = U_1^0 := U_1(0)$ for $t\ge 0$ due to the continuity of $U_1$. This last property, along with~\eqref{ls03b} and~\eqref{ls05}, entails that 
\begin{subequations}\label{ls07}
\begin{equation}
	U_i(t) = U_i^0 := U_i(0)\,, \qquad t\ge 0\,,\ 1\le i \le 4\,, \label{ls07a} 
\end{equation}
along with
\begin{equation}
	U_4^0 = \frac{k_1}{k_4} U_1^0\,, \ U_3^0 = \frac{U_2^0}{k_3} \,, \ \big(U_2^0\big)^2 \big( U_1^0 - k_2 U_2^0 \big) = 0\,, \label{ls07b}
\end{equation}
and
\begin{equation}
	\frac{k_1+k_4}{k_4} U_1^0 + \frac{1+k_3}{k_3} U_2^0 = \frac{\varrho}{|\Omega|}\,. \label{ls07c}
\end{equation}
\end{subequations}
On the one hand, if $U_2^0=0$, then $U_3^0=0$ by~\eqref{ls07b}, while \eqref{ls07} gives
\begin{equation*}
	U_1^0 = \frac{\varrho k_4}{|\Omega| (k_1+k_4)}\,, \quad U_4^0 = \frac{\varrho k_1}{|\Omega| (k_1+k_4)}.
\end{equation*}
Consequently, $\bU^0 = \mathbf{E}_{b,\varrho}$, see~\eqref{eqr}. On the other hand, if $U_2^0\ne 0$, then $U_2^0=U_1^0/k_2$ by~\eqref{ls07b} and we readily infer from~\eqref{ls07} that $\bU^0 = \mathbf{E}_{\circ,\varrho}$, see~\eqref{eqr}. This completes the proof. 
\end{proof}

\begin{cor}\label{C1a}
Let $p\in (n,\infty)$, $p\ge 2$, and $\varrho>0$. Given $\bu^0\in \mathcal{Z}_{p,\varrho}^+$, there is $*\in\{b,\circ\}$ (depending on $\bu^0$) such that
\begin{equation*}
	\lim_{t\to\infty} \|\bu(t) - \mathbf{E}_{*,\varrho} \|_p=0\,. 
\end{equation*}	
\end{cor}

\begin{proof}
According to \Cref{L1}, $\mathcal{E}_2$ is a strict Liapunov functional for the semiflow $\bu(\cdot;\bu^0)$ on $\mathcal{Z}_{p,\varrho}^+$. Thus, given $\bu^0\in \mathcal{Z}_{p,\varrho}^+$, we infer from the relative compactness of the orbit $\{\bu(t;\bu^0)\ :\ t\ge 0\}$ in $L_p(\Omega)$, see \Cref{T1}, and LaSalle's invariance principle that the $\omega$-limit set $\omega(\bu^0)$ defined by
\begin{equation*}
	\omega(\bu^0) := \bigcap_{t>0}\overline{\bu\big((t,\infty);\bu^0\big)}
\end{equation*}
is non-empty, connected and compact and included in $\big\{\mathbf{E}_{b,\varrho}\,, \mathbf{E}_{\circ,\varrho}\big\}$, see \cite[(17.2)~Theorem \& (18.3)~Theorem]{Ama1990} for instance. The connectedness of $\omega(\bu^0)$ completes the proof.
\end{proof}

\begin{cor}\label{C1}
Let $p\in (n,\infty)$, $p\ge 2$, and $\varrho>0$. Then $\mathcal{E}_2(\mathbf{E}_{\circ,\varrho}) < \mathcal{E}_2(\mathbf{E}_{b,\varrho})$ and, if $\bu^0\in \mathcal{Z}_{p,\varrho}^+$ satisfies $\mathcal{E}_2(\bu^0)<\mathcal{E}_2(\mathbf{E}_{b,\varrho})$, then $\mathbf{E}_{*,\varrho}=\mathbf{E}_{\circ,\varrho}$ in~\eqref{ltb}. 
\end{cor}

An immediate consequence of \Cref{C1} is that $\mathbf{E}_{\circ,\varrho}$ is locally asymptotically stable in $L_p(\Omega)$ for the semiflow $\bu(\cdot;\bu^0)$. We shall actually improve this result in the next section and establish its local exponential stability by means of the principle of linearized stability.

\begin{proof}[Proof of \Cref{C1}]
Owing to~\eqref{eqr}, 
\begin{align*}
	\mathcal{E}_2(\mathbf{E}_{\circ,\varrho} )& = \frac{\varrho^2}{|\Omega|} \frac{(k_2k_3k_4)^2 + k_2(k_3k_4)^2 + k_2k_3k_4^2 + k_1 (k_2k_3)^2k_4}{K_0^2} \\
	& = \frac{\varrho^2}{|\Omega|} \frac{k_2k_3k_4}{K_0} < \frac{\varrho^2}{|\Omega|} \frac{k_4}{k_1+k_4} = \mathcal{E}_2(\mathbf{E}_{b,\varrho})\,.
\end{align*}
Now, given $\bu^0\in \mathcal{Z}_{p,\varrho}^+$ satisfying $\mathcal{E}_2(\bu^0)<\mathcal{E}_2(\mathbf{E}_{b,\varrho})$, we infer from \Cref{P2}, \Cref{T2} and the semiflow properties that
\begin{equation*}
	\mathcal{E}_2(\mathbf{E}_{*,\varrho}) = \lim_{t\to\infty} \mathcal{E}_2(\bu(t)) \le \mathcal{E}_2(\bu^0) <\mathcal{E}_2(\mathbf{E}_{b,\varrho})\,,
\end{equation*} 
and $\mathbf{E}_{*,\varrho}=\mathbf{E}_{b,\varrho}$ is excluded. Consequently, $\mathbf{E}_{*,\varrho}=\mathbf{E}_{\circ,\varrho}$, as claimed.
\end{proof}

\subsection{Stability of Equilibria}\label{S3.2}

Fix $p\in (n,\infty)$ and $p\ge 2$. Setting
\begin{equation*}
    \mathbb{P}\bu:=\frac{P\bu}{4\vert\Omega\vert}\mathbf{1}\,,\quad \bu\in L_p(\Omega)\,,
\end{equation*}
with
\begin{equation*}
P\bu:=\int_\Omega \sum_{i=1}^4 u_i(x)\,\rd x\,,\qquad \mathbf{1}=(1,1,1,1)\,,
\end{equation*}
it readily follows that $\mathbb{P}\in\mathcal{L}(L_p(\Omega))$ is a projection with $\ker \mathbb{P}= \mathcal{Z}_{p,0}$. Therefore, we have the direct sum decomposition
\begin{equation*}
    L_p(\Omega)=\mathbb{P}(L_p(\Omega))\oplus \mathcal{Z}_{p,0}\,.
\end{equation*}
We then introduce 
\begin{align*}
    \mathcal{A} & :=\mathrm{diag}[d_1\Delta_N,d_2\Delta_N,d_3\Delta_N,d_4\Delta_N]\,, \\
	F(\bu) & := f(\bu) + \big(-k_1 u_1,-u_2,-k_3u_3,-k_4 u_4\big)\,, 
\end{align*}
with $f$ defined in~\eqref{deff}, so that~\eqref{U} can be written as
\begin{equation}\label{CP1}
	\frac{\rd \bu}{\rd t}=\mathcal{A}\bu+F(\bu)\,,\quad t>0\,,\qquad \bu(0)=\bu^0\,.
\end{equation}
Setting $\mathbf{v}:=(1-\mathbb{P})\bu$ so that $\bu=\mathbb{P}\bu+\mathbf{v}$ and noticing that $\mathbb{P}\mathcal{A} =\mathcal{A}\mathbb{P} = \mathbf{0}$ and that $\mathbb{P}F(\bu)=0$, it readily follows that $\mathbb{P}\bu(t)=\mathbb{P}\bu^0$. Hence, \eqref{CP1} is equivalent to the Cauchy problem 
\begin{equation}\label{CP2}
	\frac{\rd \mathbf{v}}{\rd t}=\mathcal{A}\mathbf{v}+g(\mathbf{v})\,,\quad t>0\,,\qquad \mathbf{v}(0)=(1-\mathbb{P})\bu^0\,,
\end{equation}
in~$\mathcal{Z}_{p,0}$, where 
$$
g(\mathbf{v}):=F(\mathbb{P}\bu^0+\mathbf{v})\,.
$$
Note that the operator $\mathcal{A}\vert_{((1-\mathbb{P})W_{p,N}^2(\Omega))}$ still generates a strongly continuous analytic semigroup on $\mathcal{Z}_{p,0}$ and that $g\in C^1((1-\mathbb{P})W_p^{2\xi}(\Omega),\mathcal{Z}_{p,0})$ for $\xi\in (n/(2p),1/2)$ (see the proof of Proposition~\ref{P1}).

Consider now $\varrho>0$ and $\bu^0\in \mathcal{Z}_{p,\varrho}^+$, so that 
\begin{equation*}
   P\bu^0=\varrho \;\;\text{ and }\;\; g(\mathbf{v}) = F\left( \frac{\varrho}{4|\Omega|}\mathbf{1} + \mathbf{v} \right),
\end{equation*} 
in view of $\bu(t)\in\mathcal{Z}_{p,\varrho}^+$, see \Cref{T1}. Then $\mathbf{v}_{*,\varrho}:=(1-\mathbb{P})\mathbf{E}_{*,\varrho}$ is an equilibrium to~\eqref{CP2} for $*\in\{b,\circ\}$ since $P\mathbf{E}_{*,\varrho}=\varrho$, and 
 the corresponding linearized operator 
 \begin{equation*}
    \mathcal{L}_{*,\varrho}:=\mathcal{A} +Dg(\mathbf{v}_{*,\varrho})
\end{equation*}
is given by
\begin{equation}
	\mathcal{L}_{*,\varrho}[\bv] = 
	\begin{pmatrix} 
		d_1 \Delta v_1 - \mathbf{E}_{*,\varrho,2}^2 v_1 + k_2 \mathbf{E}_{*,\varrho,2}^2 v_2 - k_1 v_1 + k_4 v_4\vspace{0.2cm} \\ 
		d_2 \Delta v_2 + \mathbf{E}_{*,\varrho,2}^2 v_1 - k_2 \mathbf{E}_{*,\varrho,2}^2 v_2 - v_2 + k_3 v_3\vspace{0.2cm} \\
		d_3 \Delta v_3 + v_2 - k_3 v_3\vspace{0.2cm} \\
		d_4 \Delta v_4 + k_1 v_1 - k_4 v_4
	\end{pmatrix}\,, \label{ll01}
\end{equation}
after using the identity $\mathbf{E}_{*,\varrho,1} \mathbf{E}_{*,\varrho,2} = k_2 \mathbf{E}_{*,\varrho,2}^2$, which is valid for both equilibria according to~\eqref{eqr}. Clearly, $\mathcal{L}_{*,\varrho}$ with domain $(1-\mathbb{P})(W_{p,N}^2(\Omega))$ generates an analytic semigroup on $\mathcal{Z}_{p,0}$ since $Dg(\mathbf{v}_{*,\varrho})\in\mathcal{L}(\mathcal{Z}_{p,0},\mathcal{Z}_{p,0})$. A straightforward computation gives the following identity: 

\begin{lem}\label{L2}
Let $*\in\{b,\circ\}$. Consider $\bv\in \mathcal{Z}_{p,0}$ (possibly complex valued) and set 
\begin{equation}
	\bV(\bv) := \left( v_1\,,k_2 v_2\,, k_2k_3v_3\,, \frac{k_4}{k_1}v_4 \right)\,. \label{ll02}
\end{equation}
Then
\begin{equation}
\begin{split}
	\int_\Omega \mathcal{L}_{*,\varrho}[\bv] \overline{\bV(\bv)}\,\rd x & = - \sum_{i=1}^4 D_i \|\nabla v_i\|_2^2 - \mathbf{E}_{*,\varrho,2}^2 \|v_1-k_2 v_2\|_2^2 \\
	& \quad - k_2 \|v_2-k_3v_3\|_2^2 - \frac{1}{k_1} \|k_1 v_1 - k_4 v_4\|_2^2 \le 0\,,
\end{split} \label{ll03}
\end{equation}
with
\begin{equation*}
	D_1 := d_1\,, \quad D_2 := d_2k_2\,, \quad D_3:= d_3k_2k_3\,, \quad D_4 := d_4 \frac{k_4}{k_1}\,.
\end{equation*}
\end{lem}

We use the just established identity to deduce information on the spectrum of $\mathcal{L}_{*,\varrho}$ in~$\mathcal{Z}_{p,0}$.

\begin{prop}\label{P3}
For $*\in\{b,\circ\}$, the operator $\mathcal{L}_{*,\varrho}$ in $\mathcal{Z}_{p,0}$ has compact resolvent. In particular, its spectrum is discrete and contains only eigenvalues. Moreover, it is included in $(-\infty,0]$. In fact, $0\in \sigma\big( \mathcal{L}_{b,\varrho} \big)$ with 
\begin{equation*}
    \ker \mathcal{L}_{b,\varrho} = \mathbb{R} \big( k_4(1+k_3), - k_3(k_1+k_4), -(k_1+k_4), k_1(1+k_3) \big)\,,
\end{equation*}
and there are $(\lambda_{b,\varrho},\lambda_{\circ,\varrho})\in (0,\infty)^2$ such that 
\begin{equation*}
    \sigma\big( \mathcal{L}_{b,\varrho} \big) \setminus \{0\}\subset (-\infty,-\lambda_{b,\varrho}] \quad \text{and}\quad \sigma\big( \mathcal{L}_{\circ,\varrho} \big) \subset (-\infty,-\lambda_{\circ,\varrho}]\,.
\end{equation*}
\end{prop}

\begin{proof}
Let $*\in\{b,\circ\}$. Since $(1-\mathbb{P})(W_{p,N}^2(\Omega))$ embeds compactly in $\mathcal{Z}_{p,0}=(1-\mathbb{P})(L_p(\Omega))$, the operator~$\mathcal{L}_{*,\varrho}$ has compact resolvent. Therefore, its spectrum $\sigma\big( \mathcal{L}_{*,\varrho} \big)$ is discrete and contains only eigenvalues. Consider then $\lambda\in \sigma\big( \mathcal{L}_{*,\varrho} \big)$ and a corresponding eigenvector $\bv\ne\mathbf{0}$. As $\mathcal{L}_{*,\varrho}[\bv] = \lambda\bv$, we infer from \Cref{L2} that
\begin{equation}
	0\ge \int_\Omega \mathcal{L}_{*,\varrho}[\bv] \overline{\bV(\bv)}\,\rd x = \lambda \left( \|v_1\|_2^2 + k_2 \|v_2\|_2^2 + k_2k_3 \|v_3\|_2^2 + \frac{k_4}{k_1} \|v_4\|_2^2 \right) \,. \label{ll04}
\end{equation}
Thus $\lambda\in (-\infty,0]$. 

Assume now that $0\in \sigma\big( \mathcal{L}_{*,\varrho} \big)$ and consider $\bv\in \ker\mathcal{L}_{*,\varrho}$. We then infer from~\eqref{ll03} and~\eqref{ll04} that
\begin{equation*}
    0 = \sum_{i=1}^4 D_i \|\nabla v_i\|_2^2 + \mathbf{E}_{*,\varrho,2}^2 \|v_1-k_2 v_2\|_2^2 + k_2 \|v_2-k_3v_3\|_2^2 + \frac{1}{k_1} \|k_1 v_1 - k_4 v_4\|_2^2\,,
\end{equation*}
from which we deduce that $\bv\in\mathbb{R}^4$ with
\begin{equation}
  \mathbf{E}_{*,\varrho,2} (v_1-k_2 v_2) = v_2-k_3v_3  = k_1 v_1 - k_4 v_4 = 0\,. \label{ll05}
\end{equation}

\noindent$\bullet$ If $*=\circ$ then, since $E_{\circ,\varrho,2}\ne 0$, an immediate consequence of~\eqref{ll05} is that
\begin{equation*}
    v_2 = \frac{v_1}{k_2}\,, \quad v_3 = \frac{v_1}{k_2k_3}\,, \quad v_4 = \frac{k_1 v_1}{k_4}\,.
\end{equation*}
Recalling that $\bv\in \mathcal{Z}_{p,0}$, we obtain
\begin{equation*}
    0 = |\Omega| \left( 1 + \frac{1}{k_2} + \frac{1}{k_2 k_3} + \frac{k_1}{k_4} \right) v_1\,,
\end{equation*}
whence $v_1=0$ and $\bv=\mathbf{0}$; that is, $0\not\in \sigma\big(\mathcal{L}_{\circ,\varrho}\big)$. Since the spectrum is discrete, there is $\lambda_{\circ,\varrho}>0$ such that 
$\sigma\big( \mathcal{L}_{\circ,\varrho} \big) \subset (-\infty,-\lambda_{\circ,\varrho}]$.

\noindent$\bullet$ If $*=b$, then $E_{b,\varrho,2}= 0$ and we only deduce from~\eqref{ll05} that $v_3 = v_2/k_3$ and $v_4 = k_1 v_1/k_4$. Since $\bv\in \mathcal{Z}_{p,0}$, we additionally obtain that
\begin{equation*}
    0 = |\Omega| \left( v_1 + v_2 + \frac{v_2}{k_3} + \frac{k_1 v_1}{k_4} \right)\,,
\end{equation*}
hence $(k_1+k_4) v_1/k_4 = -(1+k_3)v_2/k_3$. Consequently,
\begin{equation*}
    \bv = v_1 \left( 1 , - \frac{k_3(k_1+k_4)}{k_4(1+k_3)} , - \frac{k_1+k_4}{k_4(1+k_3)} , \frac{k_1}{k_4} \right)\,,
\end{equation*}
and we thus have identified $\ker\mathcal{L}_{b,\varrho}$. Since the spectrum is discrete, there is $\lambda_{b,\varrho}>0$ such that $\sigma\big( \mathcal{L}_{b,\varrho} \big)\setminus\{0\} \subset (-\infty,-\lambda_{b,\varrho}]$.
\end{proof}

After this preparation, we are ready to study more precisely the local behavior of the semiflow near the steady states and begin with the interior steady state, postponing the analysis near the boundary steady state to the next subsection. The principle of linearized stability now implies the exponential asymptotic stability of $\mathbf{E}_{\circ,\varrho}$  in $\mathcal{Z}_{p,\varrho}^+$.

\begin{cor}\label{C2}
Let $p\in (n,\infty)$, $p\ge 2$, and $\varrho>0$. Then $\mathbf{E}_{\circ,\varrho}$ is exponentially asymptotically stable in $\mathcal{Z}_{p,\varrho}^+$. That is, given $\omega\in (0,\lambda_{\circ,\varrho})$ there are $r>0$ and $M\ge 1$ such that
\begin{equation*}
    \|\bu(t;\bu^0)-\mathbf{E}_{\circ,\varrho}\|_p\le M\|\bu^0-\mathbf{E}_{\circ,\varrho}\|_p e^{-\omega t}\,,\quad t\ge 0\,,
\end{equation*}
whenever $\bu^0\in \mathcal{Z}_{p,\varrho}^+$ satisfies $\|\bu^0-\mathbf{E}_{\circ,\varrho}\|_p\le r$.
\end{cor}

\begin{proof}
The statement readily follows from Proposition~\ref{P3} and the principle of linearized stability stated in \cite[Theorem~1.6]{MW_MN} (with $\gamma=\gamma_*=\alpha=0<\xi<1/2$, $q=2$, and $\alpha_{crit}=\alpha_{crit}^*=2\xi-1<0$ therein) applied to~\eqref{CP2}.
\end{proof}

\subsection{Boundary Steady State}\label{S3.3}

According to Proposition~\ref{P3}, the principle of linearized stability does not apply to investigate the stability properties of the boundary steady state~$\mathbf{E}_{b,\varrho}$, since the corresponding linearization $\mathcal{L}_{b,\varrho}$ has eigenvalue zero. In fact, $\mathbf{E}_{b,\varrho}$ is a non-hyperbolic fixed point of the semiflow and the appropriate tool to study the local behavior of the semiflow in the vicinity of such a point is the center manifold theory \cite{Carr81}. As a first step towards the construction of a (local) center manifold, let us recast~\eqref{CP2} in a suitable way. Introducing
\begin{equation*}
    \mathbf{w}:=\mathbf{v}-(1-\mathbb{P})\mathbf{E}_{b,\varrho}=(1-\mathbb{P})(\bu-\mathbf{E}_{b,\varrho})\,,
\end{equation*}
we infer from~\eqref{CP2} that
\begin{equation}\label{CP2w}
\begin{split}
	& \frac{\rd \mathbf{w}}{\rd t} = \mathcal{L}_{b,\varrho}[\mathbf{w}]+ N(\mathbf{w}) \,,\quad t>0\,, \\
    & \mathbf{w}(0) = \mathbf{w}^0 := (1-\mathbb{P})(\bu^0-\mathbf{E}_{b,\varrho}) \,,
\end{split}
\end{equation}
in $\mathcal{Z}_{p,0}$, where
\begin{equation*}
    \mathcal{L}_{b,\varrho}[\mathbf{w}]=\mathcal{A}\mathbf{w} +DF(\mathbf{E}_{b,\varrho})[\mathbf{w}]
\end{equation*}
and
\begin{equation*}
    N(\mathbf{w}):=F\big(\mathbf{E}_{b,\varrho}+\mathbf{w}\big)-DF(\mathbf{E}_{b,\varrho})[\mathbf{w}]\,.
\end{equation*}
We recall that $\mathcal{L}_{b,\varrho}$ with domain $(1-\mathbb{P})W_{p,N}^2(\Omega)$ generates an analytic semigroup on $\mathcal{Z}_{p,0}$ and has the one-dimensional kernel $\ker \mathcal{L}_{b,\varrho} = \mathbb{R} \mathbf{k}$
with
\begin{equation}\label{k}
    \mathbf{k}:=\big( k_4(1+k_3), - k_3(k_1+k_4), -(k_1+k_4), k_1(1+k_3) \big)
\end{equation}
by Proposition~\ref{P3}. Moreover,
\begin{equation*}
    N(\mathbf{0})= \mathbf{0}\,,\qquad DN(\mathbf{0})= \mathbf{0}\,,
\end{equation*}
and, in fact,
\begin{equation}\label{N}
	N(\mathbf{w}) = w_2^2\, \big(
		k_2 w_2- \varrho K_1 k_4-w_1\,,\, 
			-k_2 w_2+\varrho K_1 k_4+w_1\,,\, 0\,,\,0\big)
\end{equation}
with $K_1:=1/[\vert\Omega\vert (k_1+k_4)]$. To construct a splitting of the space $\mathcal{Z}_{p,0}$ into two $\mathcal{L}_{b,\varrho}$-invariant subspaces according to the zero eigenvalue we introduce
\begin{equation}\label{q}
q(\mathbf{w}):=K_2 \int_\Omega\big(w_1+w_4-w_2-w_3\big)\,\rd x\,,\quad \mathbf{w}\in \mathcal{Z}_{p,0}\,,
\end{equation}
with
\begin{equation*}
    K_2 :=\frac{1}{2\vert\Omega\vert (k_1+k_4)(1+k_3)}
\end{equation*}
and note that $q(\mathbf{k})=1$ and $q(\mathcal{L}_{b,\varrho}[\mathbf{w}])=0$ (see \eqref{ll01} and \eqref{eqr}). Therefore,
\begin{equation*}
    Q\mathbf{w}:= q(\mathbf{w})\mathbf{k}\,,\quad \mathbf{w}\in \mathcal{Z}_{p,0}\,,
\end{equation*}
defines a projection $Q=Q^2\in\mathcal{L}(\mathcal{Z}_{p,0})$ with $X:=\mathrm{rg}Q= \mathbb{R} \mathbf{k}$ and $Q\mathcal{L}_{b,\varrho}=\mathcal{L}_{b,\varrho}Q=0$. Setting $Y:=\mathrm{rg}(1-Q)$ we thus obtain the decomposition
\begin{equation}
    \mathcal{Z}_{p,0}= \mathbb{R} \mathbf{k}\oplus Y=X\oplus Y \label{dec}
\end{equation}
with
$\mathcal{L}_{b,\varrho}\vert_X=0$ and $\mathcal{B}:=\mathcal{L}_{b,\varrho}(1-Q)=\mathcal{L}_{b,\varrho}\vert_Y$ generates an analytic semigroup on $Y$. Moreover, its spectrum is included in $(-\infty,-\lambda_{b,\varrho}]$ according to \Cref{P3}. Consequently, assumptions (i)-(iii) of \cite[Section~6.3]{Carr81} hold. 

Using the decomposition $\mathbf{w}=s\mathbf{k}+\mathbf{y}$ with $s\in \R$ and $\mathbf{y}\in Y$, we can write~\eqref{CP2w} equivalently in the form
\begin{subequations}
\label{deco}
\begin{align} 
\frac{\rd s}{\rd t} &= q\left( N\big(s\mathbf{k}+\mathbf{y}\big) \right)\,,\\
\frac{\rd \mathbf{y}}{\rd t}&=\mathcal{B}[\mathbf{y}]+ (1-Q)N\big(s\mathbf{k}+\mathbf{y}\big)\,,\\
(s,\mathbf{y})(0) & = \big(s^0,\mathbf{y}^0\big)\in \mathbb{R}\times Y\,,
\end{align}
\end{subequations}
(where $\mathbf{w}^0 = s^0 \mathbf{k} + \mathbf{y}^0$). It readily follows from \Cref{T1} that~\eqref{deco} is well-posed in $\mathbb{R}\times Y$ and that, owing to the identity
\begin{equation*}
(1-\mathbb{P})\big(\mathbf{u}\big(t;\mathbf{u}^0\big) - \mathbf{E}_{b,\varrho}\big)=\mathbf{u}\big(t;\mathbf{u}^0\big) - \mathbf{E}_{b,\varrho}\,,\qquad \mathbf{u}^0\in \mathcal{Z}_{p,\varrho}\,,
\end{equation*}
the mapping 
\begin{equation}
    \big(t,s^0,\mathbf{y}^0\big)\mapsto \mathbf{w}\big(t;s^0,\mathbf{y}^0\big) := \mathbf{u}\big(t;\mathbf{E}_{b,\varrho} + s^0 \mathbf{k} + \mathbf{y}^0\big) - \mathbf{E}_{b,\varrho} \label{wflow}
\end{equation}
defines a global semiflow on $\mathbb{R}\times Y$. We may now state the main result of this section.

\begin{prop}\label{P5}
There exists a $C^2$-smooth (local) center manifold $\mathcal{W}_c(\mathbf{E_{b,\varrho}})$ for~\eqref{deco}; that is, $\mathcal{W}_c(\mathbf{E_{b,\varrho}})$ is positively invariant for the semiflow $\mathbf{w}\big(\cdot;s^0,\mathbf{y}^0\big)$, and there are $\delta>0$ and $\mathbf{h}\in C^2\big((-\delta,\delta),Y\big)$ satisfying $\mathbf{h}(0)=\mathbf{0}$ and $\mathbf{h}'(0)=\mathbf{0}$ such that
\begin{equation*}
    \mathcal{W}_c(\mathbf{E_{b,\varrho}}) = \big\{ \xi\mathbf{k} + \mathbf{h}(\xi)\ :\ \xi\in (-\delta,\delta) \big\}\,.
\end{equation*} 
Moreover, there are $\delta_0\in (0,\delta)$ and $K_4>0$ such that
\begin{equation}
	-3K_4 \xi^2 \le q\big(N\big(\xi\mathbf{k}+\mathbf{h}(\xi)\big)\big) \le -K_4 \xi^2\,, \qquad \xi\in (-\delta_0,\delta_0)\,. \label{CM04}
\end{equation} 
\end{prop}

As we shall see below, the local behavior~\eqref{CM04} of $\xi\mapsto q\big(N\big(\xi\mathbf{k}+\mathbf{h}(\xi)\big)\big)$ is the key tool to shed some light on the dynamics of the semiflow $\mathbf{w}\big(\cdot;s^0,\mathbf{y}^0\big)$ on the center manifold $\mathcal{W}_c(\mathbf{E_{b,\varrho}})$.

\begin{proof}
As already mentioned, assumptions (i)-(iii) of \cite[Section~6.3]{Carr81} are satisfied and the existence of a $C^2$-smooth (local) center manifold $\mathcal{W}_c(\mathbf{E_{b,\varrho}})$ for~\eqref{deco} readily follows from \cite[6.Theorem~8]{Carr81}.

We next turn to a detailed study of the behavior of $\xi\mapsto q\big(N\big(\xi\mathbf{k}+\mathbf{h}(\xi)\big)\big)$ as $\xi\to 0$ and prove~\eqref{CM04}. We first observe that the properties of $\mathbf{h}$ guarantee that there is $H>0$ such that
\begin{equation}
	\|\mathbf{h}(\xi)\|_p \le H \xi^2\,, \qquad \xi\in (-\delta,\delta)\,. \label{CM03}
\end{equation}
Next, from~\eqref{eqr}, \eqref{N} and~\eqref{q} we obtain that, for $\xi\in (-\delta,\delta)$,
\begin{align*}
	& q\big(N\big(\xi\mathbf{k}+\mathbf{h}(\xi)\big)\big) \\
	& \quad = -2K_2 \int_\Omega \big[- k_3(k_1+k_4) \xi + h_2(\xi)\big]^2 \big[ \varrho K_1 k_4 + K_0 \xi + h_1(\xi) - \xi h_2(\xi) \big]\,\rd x \\
	& \quad = - 2 \varrho K_1 K_2 k_4 \int_\Omega \big[- k_3(k_1+k_4) \xi + h_2(\xi)\big]^2\,\rd x \\
	& \quad\quad - 2K_2 \int_\Omega \big[- k_3(k_1+k_4) \xi + h_2(\xi)\big]^2 \big[ K_0 \xi + h_1(\xi) - \xi h_2(\xi) \big]\,\rd x \\
	& \quad = - 2 \varrho K_1 K_2 k_3^2 k_4 (k_1+k_4)^2 |\Omega| \xi^2 \\
	& \quad\quad + 4 \varrho K_1 K_2 k_3 k_4 (k_1+k_4) \int_\Omega \xi h_2(\xi)\, \rd x - 2 \varrho K_1 K_2 k_4 \int_\Omega |h_2(\xi)|^2\, \rd x \\
	& \quad\quad - 2K_2 \int_\Omega \big[- k_3(k_1+k_4) \xi + h_2(\xi)\big]^2 \big[ K_0 \xi + h_1(\xi) - \xi h_2(\xi) \big]\,\rd x\,.
\end{align*}
Setting $K_4 := \varrho K_1 K_2 k_3^2 k_4 (k_1+k_4)^2 |\Omega|>0$ and recalling that $p\ge 3$, we further infer from~\eqref{CM03} and H\"older's inequality that
\begin{align*}
	& \left| q\big(N\big(\xi\mathbf{k}+\mathbf{h}(\xi)\big)\big) + 2 K_4 \xi^2 \right| \\
	& \qquad \le 4 \varrho K_1 K_2 k_3 k_4 (k_1+k_4) |\xi| \|h_2(\xi)\|_1 + 2 \varrho K_1 K_2 k_4 \|h_2(\xi)\|_2^2\\
	& \qquad\quad + 2K_2 \|k_3(k_1+k_4) \xi - h_2(\xi)\|_3^2 \| K_0 \xi + h_1(\xi) - \xi h_2(\xi)\|_3 \\
	& \qquad \le K_5 |\xi|^3\,,
\end{align*}
for some constant $K_5>0$ depending only on $(k_i)_{1\le i\le 4}$, $|\Omega|$, $\varrho$, $p$ and $H$. In particular, there is $\delta_0\in (0,\delta)$ such that~\eqref{CM04} holds true.
\end{proof}

A first consequence of \Cref{P5} is the convergence of $\mathbf{w}\big(t;s^0,\mathbf{h}(s^0)\big)$ to zero when $s^0\in (0,\delta_0)$. However, this result is irrelevant for the identification of the long-term limit of non-negative solutions to~\eqref{U} as the corresponding initial values $\bu^0 = \mathbf{E}_{b,\varrho} + s^0 \mathbf{k} + \mathbf{h}(s^0)$ have non-positive second and third components for sufficiently small $s^0>0$ (see~\eqref{eqr} and~\eqref{k}). 

\begin{cor}\label{C3}
If $s^0\in (0,\delta_0)$, then 
\begin{equation}
    \mathbf{w}\big([0,\infty);s^0,\mathbf{h}(s^0)\big)\subset \mathcal{W}_c(\mathbf{E_{b,\varrho}}) \;\;\text{ and }\;\; \lim_{t\to\infty} \big\|\mathbf{w}\big(t;s^0,\mathbf{h}(s^0)\big)\big\|_{\mathbb{R}\times Y} = 0\,.\label{cvbss}
\end{equation} 
\end{cor}

\begin{proof}
 We consider $s^0\in (0,\delta_0)$ and set $\mathbf{w}(t) := \mathbf{w}\big(t;s^0,\mathbf{h}(s^0)\big)$ for $t\ge 0$. Owing to the positive invariance of $\mathcal{W}_c(\mathbf{E_{b,\varrho}})$ for the semiflow, there is $\tau\in (0,\infty]$ such that $\mathbf{w}(t)\in \mathcal{W}_c(\mathbf{E_{b,\varrho}})$ for all $t\in [0,\tau)$. Consequently, there is $s\in C^1([0,\tau))$ such that
\begin{equation}
	s(t)\in (-\delta,\delta) \;\;\text{ and }\;\; \mathbf{w}(t) = s(t)\mathbf{k} + \mathbf{h}(s(t))\,, \qquad t\in [0,\tau)\,. \label{CM01}
\end{equation}
Combining~\eqref{deco} and~\eqref{CM01} readily implies that $s$ solves
\begin{equation}
\frac{\rd s}{\rd t}= q\big(N\big(s\mathbf{k}+\mathbf{h}(s)\big)\big)\,, \quad t\in [0,\tau)\,, \qquad s(0)=s^0\,. \label{CM02}
\end{equation}
We next define 
\begin{equation*}
	\tau_0 := \inf\big\{ t\in [0,\tau)\ :\ s(t)\not\in (0,\delta_0) \big\}\in (0,\tau]\,, 
\end{equation*}
the positivity of $\tau_0$ being an immediate consequence of $s^0\in (0,\delta_0)$ and the continuity of $s$. Owing to~\eqref{CM02} and~\eqref{CM04}, 
\begin{equation*}
	- 3 K_4 s^2(t) \le \frac{\rd s}{\rd t}(t) \le - K_4 s^2(t)\,, \qquad t\in [0,\tau_0)\,,
\end{equation*}
from which we readily deduce that
\begin{equation*}
	0 < \frac{s^0}{1+3K_4 s^0 t} \le s(t) \le \frac{s^0}{1+K_4 s^0 t} < \delta_0\,, \qquad t\in [0,\tau_0)\,.
\end{equation*}
Therefore, $\tau_0=\tau=\infty$ and $s(t)\in [0,\delta_0)$ for all $t\ge 0$ with $s(t)\to 0$ as $t\to\infty$, the latter entailing that $\|\mathbf{h}(s(t))\|_Y\to 0$ as $t\to\infty$ as well. Recalling~\eqref{CM01}, we have shown that, for $s^0\in (0,\delta_0)$, $\mathbf{w}\big(\cdot;s^0,\mathbf{h}(s^0)\big)$ satisfies~\eqref{cvbss}. 
\end{proof}

Summarizing our findings from this section we obtain \Cref{T2}.

\begin{proof}[Proof of \Cref{T2}]
Let $\varrho>0$ and $\mathbf{u}^0\in\mathcal{Z}_{p,\varrho}^+$. In \Cref{C1a} we established~\eqref{ltb}, while the exponential asymptotic stability of $\mathbf{E}_{\circ,\varrho}$ is shown in \Cref{C2}.

Assume next that $\mathbf{E}_{*,\varrho}=\mathbf{E}_{b,\varrho}$ in~\eqref{ltb}. Thanks to the decomposition~\eqref{dec}, we may write $\bu^0 = \mathbf{E}_{b,\varrho} + s^0 \mathbf{k} + \mathbf{y}^0$ and $\bu\big(\cdot;\bu^0\big) = \mathbf{E}_{b,\varrho} + s\mathbf{k} + \mathbf{y}$. The convergence~\eqref{ltb} then reads
\begin{equation}
    \lim_{t\to\infty} \left( |s(t)| + \|\mathbf{y}(t)\|_p \right) = 0\ . \label{CM05}
\end{equation}
We also infer from \cite[2.~Lemma~1]{Carr81} (the proof in the infinite-dimensional case is exactly the same) that there are $t_0\ge 0$, $K_6>0$ and $\mu>0$ such that
\begin{equation}
    \|\mathbf{y}(t) - \mathbf{h}(s(t))\|_p \le K_6 e^{-\mu(t-t_0)} \|\mathbf{y}(t_0) - \mathbf{h}(s(t_0))\|_p\ , \qquad t\ge t_0\ . \label{CM06}
\end{equation}
Combining~\eqref{CM05} and~\eqref{CM06}, we conclude that
\begin{equation}
    \lim_{t\to\infty} \left( |s(t)| + \|\mathbf{h}(s(t))\|_p \right) = 0\ .\label{CM07}
\end{equation}
Consequently, there is $t_1\ge t_0$ such that $|s(t)|<\delta_0$ for $t\ge t_1$ and it follows from~\eqref{CM04} that 
\begin{equation}
    - 3 K_4 s^2(t) \le \frac{\rd s}{\rd t}(t) \le - K_4 s^2(t)\,, \qquad t\ge t_1\,. \label{CM08}
\end{equation}
Assume for contradiction that $s(t_1)\ne 0$. Then~\eqref{CM08} implies that
\begin{equation*}
    K_4(t-t_1) + \frac{1}{s(t_1)} \le \frac{1}{s(t)} \le 3K_4(t-t_1) + \frac{1}{s(t_1)}\ , \qquad t\ge t_1\ ,
\end{equation*}
from which we deduce that $s(t)>0$ for all $t\ge t_1$. Consequently, since ${\bf u}(t,{\bf u}^0)\in \mathcal{Z}_{p,\varrho}^+$, we obtain
\begin{equation}
    0 \le \int_\Omega (u_2(t)+u_3(t))\,\rd x = -(1+k_3)(k_1+k_4)|\Omega| s(t) +\int_\Omega (y_2(t)+y_3(t))\,\rd x \label{CM09}
\end{equation}
for $t\ge t_1$. Now, since $\mathbf{y}(t)\in Y$, we have $\mathbb{P}\mathbf{y}(t)=Q\mathbf{y}(t)=0$, from which we deduce that
\begin{equation}
   \int_\Omega (y_1(t)+y_4(t))\,\rd x = \int_\Omega (y_2(t)+y_3(t))\,\rd x = 0\ , \qquad t\ge t_1\ . \label{CM10}
\end{equation}
Combining~\eqref{CM09} and~\eqref{CM10} implies that $-s(t)\ge 0$ for $t\ge t_1$ and contradicts the already established positivity of $s$ on this interval. Therefore, $s(t)=0$ for $t\ge t_1$ and we use again~\eqref{CM09} and~\eqref{CM10} to conclude that
\begin{equation*}
  \int_\Omega (u_2(t)+u_3(t))\,\rd x = 0\ , \qquad t\ge t_1\ .  
\end{equation*}
Recalling that both $u_2$ and $u_3$ are non-negative, we deduce that \begin{equation}
    u_2(t)=u_3(t)\equiv 0\,, \qquad t\ge t_1\,. \label{CM11}
\end{equation}
Now, since $u_2$ is non-negative, the comparison principle and~\eqref{u3} ensure that
\begin{equation*}
    u_3(t) \ge e^{k_3(t-s)} e^{d_3(t-s)\Delta_N} u_3(s) \ge 0\,, \qquad t\ge s\ge 0\,.
\end{equation*}
Combining~\eqref{CM11} and the above inequality with $t=t_1$ implies that $u_3(s)=0$ for $s\in [0,t_1]$. Therefore, $u_3\equiv 0$ on $[0,\infty)\times\Omega$ and we infer from this property and~\eqref{u3} that $u_2\equiv 0$ on $[0,\infty)\times\Omega$. We have thus proved that $u_2^0=u_3^0=0$. The converse statement is a direct consequence of \Cref{R1}.
\end{proof}

\section{Convergence to the Classical Gray-Scott Model}\label{S4}

This section is devoted to the proof of \Cref{T3}. We thus assume that
\begin{equation*}
    k_2=k_3=k_4=d_4=\varepsilon
\end{equation*}
for some $\varepsilon\in (0,1)$, still with $(k_1,d_1,d_2,d_3)\in (0,\infty)^4$. We consider $(u_1^0,u_2^0,u_3^0,a)\in W_{p_0}^{1,+}(\Omega)$ for some $p_0\in (n,\infty)$, $p_0\ge 2$, and denote the solution to~\eqref{U} with initial value $\bu_\varepsilon^0:=(u_1^0,u_2^0,u_3^0,a/\varepsilon)$ by $\bu_\varepsilon$. Introducing 
\begin{equation*}
    \bv_\varepsilon:=\big( u_{1,\varepsilon},u_{2,\varepsilon},u_{3,\varepsilon}, \varepsilon u_{4,\varepsilon} \big)\,,
\end{equation*}
it follows from~\eqref{U} that $\bv_\varepsilon$ solves    
\begin{subequations}\label{Ve}
\begin{align}
\partial_t v_{1,\varepsilon} &= d_1\Delta v_{1,\varepsilon} - v_{1,\varepsilon} v_{2,\varepsilon}^2 + \varepsilon v_{2,\varepsilon}^3 -k_1 v_{1,\varepsilon} + v_{4,\varepsilon}\,, & (t,x)\in (0,\infty)\times\Omega\,,\label{v1e}\\
\partial_t v_{2,\varepsilon} &= d_2\Delta v_{2,\varepsilon} + v_{1,\varepsilon} v_{2,\varepsilon}^2 - \varepsilon v_{2,\varepsilon}^3 - v_{2,\varepsilon} + \varepsilon v_{3,\varepsilon}\,, & (t,x)\in (0,\infty)\times\Omega\,,\label{v2e}\\
\partial_t v_{3,\varepsilon} &= d_3\Delta v_{3,\varepsilon} + v_{2,\varepsilon} - \varepsilon v_{3,\varepsilon}\,, & (t,x)\in (0,\infty)\times\Omega\,,\label{v3e}\\
\partial_t v_{4,\varepsilon} &= \varepsilon \Delta v_{4,\varepsilon} + \varepsilon k_1 v_{1,\varepsilon} - \varepsilon v_{4,\varepsilon}\,, & (t,x)\in (0,\infty)\times\Omega\,,\label{v4e}
\end{align}
supplemented with homogeneous Neumann boundary conditions
\begin{equation}\label{Vebc)}
    \partial_\nu v_{1,\varepsilon}=\partial_\nu v_{2,\varepsilon}=\partial_\nu v_{3,\varepsilon}=\partial_\nu v_{4,\varepsilon}=0\,,\quad  (t,x)\in (0,\infty)\times\partial\Omega\,,
\end{equation}
and initial conditions
\begin{equation}\label{Veic}
   \bv_\varepsilon(0)=\bv_\varepsilon^0 := \big(u_1^0,u_2^0,u_3^0,a\big)\,,\quad  x\in \Omega\,.
\end{equation}
\end{subequations}

In the following, we denote positive constants that are independent of $\varepsilon\in (0,1)$ but depend on $T>0$ by $c_i(T)$, $i\ge 1$.

We begin with the derivation of $L_1$-estimates on $\bv_\varepsilon$.

\begin{lem}\label{L3}
Given $T>0$, there is $c_1(T)>0$ such that
\begin{equation}
    \|v_{1,\varepsilon}(t)\|_1 + \|v_{2,\varepsilon}(t)\|_1 + \|v_{3,\varepsilon}(t)\|_1 +\|v_{4,\varepsilon}(t)\|_1 \le c_1(T)\,, \qquad t\in [0,T]\,. \label{il10}
\end{equation}
\end{lem}

\begin{proof}
We integrate~\eqref{v1e}, \eqref{v2e}, \eqref{v3e}, and~\eqref{v4e} over $\Omega$ and sum up the resulting identities to obtain from the positivity of $v_{i,\varepsilon}$ that
\begin{align*}
    \frac{\rd}{\rd t} \left( \sum_{i=1}^4 \|v_{i,\varepsilon}(t)\|_1 \right) & = (1-\varepsilon) \|v_{4,\varepsilon}\|_1 - (1-\varepsilon)k_1 \|v_{1,\varepsilon}\|_1 \le \|v_{4,\varepsilon}\|_1\,.
\end{align*}
\Cref{L3} now readily follows by Gronwall's lemma.
\end{proof}

We next turn to $L_2$-estimates.

\begin{lem}\label{L4}
Given $T>0$, there is $c_2(T)>0$ such that
\begin{equation}
    \|v_{1,\varepsilon}(t)\|_2 + \sqrt{\varepsilon} \|v_{2,\varepsilon}(t)\|_2 + \|v_{4,\varepsilon}(t)\|_2 \le c_2(T)\,, \qquad t\in [0,T]\,, \label{il11}
\end{equation}
 and
\begin{equation}
    \int_0^T \left[ d_1 \|\nabla v_{1,\varepsilon}(s)\|_2^2 + \varepsilon \|\nabla v_{4,\varepsilon}(s)\|_2^2 + \|v_{2,\varepsilon}(s) \big(v_{1,\varepsilon} - \varepsilon v_{2,\varepsilon} \big)(s)\|_2^2 \right]\,\rd s \le c_2(T)\,.
   \label{il12}  
\end{equation}
\end{lem}

\begin{proof}
It follows from~\eqref{Ve} and Young's inequality that
\begin{align*}
    & \frac{1}{2} \frac{\rd}{\rd t} \left( \|v_{1,\varepsilon}\|_2^2 + \varepsilon \|v_{2,\varepsilon}\|_2^2 + \varepsilon^2 \|v_{3,\varepsilon}\|_2^2 + \|v_{4,\varepsilon}\|_2^2 \right) \\
    & \quad = - d_1 \|\nabla v_{1,\varepsilon}\|_2^2 - \int_\Omega v_{1,\varepsilon} v_{2,\varepsilon}^2 \big( v_{1,\varepsilon} - \varepsilon v_{2,\varepsilon} \big)\,\rd x - k_1 \|v_{1,\varepsilon}\|_2^2 + \int_\Omega v_{1,\varepsilon} v_{4,\varepsilon}\,\rd x \\
    & \qquad - d_2 \varepsilon \|\nabla v_{2,\varepsilon}\|_2^2 + \varepsilon \int_\Omega v_{2,\varepsilon}^3 \big( v_{1,\varepsilon} - \varepsilon v_{2,\varepsilon} \big)\,\rd x - \varepsilon \int_\Omega v_{2,\varepsilon} \big( v_{2,\varepsilon} - \varepsilon v_{3,\varepsilon} \big)\,\rd x \\
    & \qquad - d_3 \varepsilon^2 \|\nabla v_{3,\varepsilon}\|_2^2 +\varepsilon^2 \int_\Omega v_{3,\varepsilon} \big( v_{2,\varepsilon} - \varepsilon v_{3,\varepsilon} \big)\,\rd x - \varepsilon \|\nabla v_{4,\varepsilon}\|_2^2 \\
    & \qquad + k_1\varepsilon \int_\Omega v_{1,\varepsilon} v_{4,\varepsilon}\,\rd x - \varepsilon \|v_{4,\varepsilon}\|_2^2\\
    & \quad \le - d_1 \|\nabla v_{1,\varepsilon}\|_2^2 - \big\| v_{2,\varepsilon} \big( v_{1,\varepsilon} - \varepsilon v_{2,\varepsilon} \big) \big\|_2^2 - \varepsilon \|\nabla v_{4,\varepsilon}\|_2^2 + (1+k_1\varepsilon) \int_\Omega v_{1,\varepsilon} v_{4,\varepsilon}\,\rd x \\
    & \quad \le - d_1 \|\nabla v_{1,\varepsilon}\|_2^2 - \big\| v_{2,\varepsilon} \big( v_{1,\varepsilon} - \varepsilon v_{2,\varepsilon} \big) \big\|_2^2 - \varepsilon \|\nabla v_{4,\varepsilon}\|_2^2 + \frac{1+k_1}{2} \left( \|v_{1,\varepsilon}\|_2^2 + \|v_{4,\varepsilon}\|_2^2 \right)\,.
\end{align*}  
Applying Gronwall's lemma completes the proof.
\end{proof}

At this point, we notice that, since
\begin{align*}
    \left\| v_{1,\varepsilon} v_{2,\varepsilon}^2 - \varepsilon v_{2,\varepsilon}^3 \right\|_{L_1((0,T)\times\Omega)} & = \int_0^T \int_\Omega v_{2,\varepsilon} \big( v_{1,\varepsilon} - \varepsilon v_{2,\varepsilon} \big) v_{2,\varepsilon}\,\rd x \\
    & \le \left\| v_{2,\varepsilon} \big( v_{1,\varepsilon} - \varepsilon v_{2,\varepsilon} \big) \right\|_{L_2((0,T)\times\Omega)} \big\| v_{2,\varepsilon} \big\|_{L_2((0,T)\times\Omega)}
\end{align*}
by H\"older's inequality, \Cref{L4} does not provide enough valuable information on the right-hand side of~\eqref{v1e} and~\eqref{v2e}. However, we observe that an $L_2$-estimate on $\big( v_{2,\varepsilon} \big)_{\varepsilon\in (0,1)}$ would be sufficient to obtain an $L_1$-control on the nonlinearity $\big( v_{1,\varepsilon} v_{2,\varepsilon}^2 - \varepsilon v_{2,\varepsilon}^3 \big)_{\varepsilon\in (0,1)}$, which then guarantees the compactness of $\big( v_{1,\varepsilon} \big)_{\varepsilon\in (0,1)}$ and $\big( v_{2,\varepsilon} \big)_{\varepsilon\in (0,1)}$ in $L_1((0,T)\times\Omega)$, but not that of $\big( v_{1,\varepsilon} v_{2,\varepsilon}^2 - \varepsilon v_{2,\varepsilon}^3 \big)_{\varepsilon\in (0,1)}$. Higher integrability estimates are thus required that we derive now by exploiting the structure of~\eqref{Ve} and using the following improved duality results \cite{CDF2014,DeTr2015}:

\begin{prop}\label{P4}
Given $T>0$ and $M\in L_\infty((0,T)\times\Omega)$ satisfying
\begin{equation*}
    0 < \alpha \le M \le \beta \;\text{ a.e. in }\; (0,T)\times\Omega
\end{equation*}    
for some $(\alpha,\beta)\in (0,\infty)^2$, there is $r_0\in (1,2)$ depending only on $n$, $\Omega$, $\alpha$, and $\beta$ with the following properties: for each $r\in [r_0,2]$, there is a positive constant $C_D(r)>0$ depending only on $n$, $\Omega$, $\alpha$, $\beta$, and $r$ (but not on $T$) such that, given $f\in L_r((0,T)\times\Omega)$, the solution $\xi_{M,f}$ to the backward linear parabolic initial boundary value problem
\begin{equation}
\begin{split}
    \partial_t \xi_{M,f} + M \Delta \xi_{M,f} & = f \;\;\text{ in }\;\; (0,T)\times\Omega\,, \\
    \partial_\nu \xi_{M,f} & = 0 \;\;\text{ on }\;\; (0,T)\times\partial\Omega\,, \\
    \xi_{M,f}(T) & = 0 \;\;\text{ in }\;\; \Omega\,,
\end{split}\label{dIBVP}
\end{equation}
satisfies
\begin{align*}
    & \|\Delta\xi_{M,f}\|_{L_r((0,T)\times\Omega)} \le C_D(r) \|f\|_{L_r((0,T)\times\Omega)}\,, \\
    & \|\xi_{M,f}(t)\|_r \le (1+\beta C_D(r)) (T-t)^{(r-1)/r} \|f\|_{L_r((0,T)\times\Omega)}\,, \qquad t\in [0,T]\,.
\end{align*}
\end{prop}

We refer to \cite[Section~2]{CDF2014} and \cite[Lemma~4]{DeTr2015} for a proof of \Cref{P4}.

\begin{lem}\label{L5}
Given $T>0$, there are $q>2$ (not depending on $T$) and $c_3(T)>0$ such that
\begin{align}
    \|v_{1,\varepsilon}\|_{L_q((0,T)\times\Omega)} + \|v_{2,\varepsilon}\|_{L_q((0,T)\times\Omega)} + \|v_{3,\varepsilon}\|_{L_q((0,T)\times\Omega)} & \le c_3(T)\,, \label{il13}\\
    \left\| v_{1,\varepsilon} v_{2,\varepsilon}^2 - \varepsilon v_{2,\varepsilon}^3 \right\|_{L_{2q/(q+2)}((0,T)\times\Omega)} & \le c_3(T)\,. \label{il14}
\end{align}
\end{lem}

\begin{proof}
Introducing $\Sigma_\varepsilon := v_{1,\varepsilon} + v_{2,\varepsilon} + v_{3,\varepsilon}$ and 
\begin{equation*}
    M_\varepsilon := \frac{d_1 v_{1,\varepsilon} + d_2 v_{2,\varepsilon} + d_3 v_{3,\varepsilon}}{v_{1,\varepsilon} + v_{2,\varepsilon} + v_{3,\varepsilon}} \in \big[ \min\{d_1,d_2,d_3\},\max\{d_1,d_2,d_3\} \big]\,,
\end{equation*}
we note that $M_\varepsilon$ satisfies the assumptions of \Cref{P4} with $\alpha = \min\{d_1,d_2,d_3\}>0$ and $\beta=\max\{d_1,d_2,d_3\}$. We next fix $r\in [r_0,2]$ such that
\begin{equation}\label{r}
    r \ge \max\left\{ \frac{p_0}{p_0-1} , \frac{2n+4}{n+4} \right\} > \frac{2n}{n+2}
\end{equation}
with $r_0$ given by \Cref{P4} and deduce from~\eqref{Ve} and~\eqref{dIBVP}, for $f\in L_r((0,T)\times\Omega)$, that
\begin{align*}
    \frac{\rd}{\rd t} \int_\Omega \Sigma_\varepsilon \xi_{M_\varepsilon,f}\,\rd x & = \int_\Omega \Sigma_\varepsilon \big( -M_\varepsilon \Delta\xi_{M_\varepsilon,f} +f \big)\,\rd x + \int_\Omega \xi_{M_\varepsilon,f} \big[ \Delta(M_\varepsilon\Sigma_\varepsilon) + v_{4,\varepsilon} - k_1 v_{1,\varepsilon} \big] \,\rd x\\
    & = \int_\Omega \big[ \nabla(M_\varepsilon \Sigma_\varepsilon)\cdot \nabla \xi_{M_\varepsilon,f} + f \Sigma_\varepsilon \big]\,\rd x \\
    & \quad - \int_\Omega \left[ \nabla \xi_{M_\varepsilon,f}\cdot\nabla(M_\varepsilon \Sigma_\varepsilon) + (k_1 v_{1,\varepsilon} - v_{4,\varepsilon}) \xi_{M_\varepsilon,f}\right]\,\rd x\\
    & = \int_\Omega f \Sigma_\varepsilon\,\rd x - \int_\Omega (k_1 v_{1,\varepsilon} - v_{4,\varepsilon}) \xi_{M_\varepsilon,f}\,\rd x\,.
\end{align*}
Integrating over $(0,T)$, we end up with
\begin{equation*}
    -\int_\Omega \Sigma_\varepsilon(0,x) \xi_{M_\varepsilon,f}(0,x)\,\rd x = \int_0^T \int_\Omega f \Sigma_\varepsilon\,\rd x\rd t - \int_0^T \int_\Omega (k_1 v_{1,\varepsilon} - v_{4,\varepsilon}) \xi_{M_\varepsilon,f}\,\rd x\rd t\,, 
\end{equation*}
from which we deduce that
\begin{equation}
\begin{split}
    \left| \int_0^T \int_\Omega f \Sigma_\varepsilon\,\rd x\rd t \right| & \le \left| \int_\Omega \Sigma_\varepsilon(0,x) \xi_{M_\varepsilon,f}(0,x)\,\rd x \right| \\
    & \quad + \left| \int_0^T \int_\Omega (k_1 v_{1,\varepsilon} - v_{4,\varepsilon}) \xi_{M_\varepsilon,f}\,\rd x\rd t \right| \,.
\end{split} \label{il15}
\end{equation}
On the one hand, by H\"older's inequality, \Cref{P4}, and~\eqref{r} we have
\begin{align}
    & \left| \int_\Omega \Sigma_\varepsilon(0,x) \xi_{M_\varepsilon,f}(0,x)\,\rd x \right| \le \|\Sigma_\varepsilon(0)\|_{r/(r-1)} \|\xi_{M_\varepsilon,f}(0)\|_r \nonumber \\
    & \qquad \le \big( 1 + \max\{d_1,d_2,d_3\} C_D(r) \big) T^{(r-1)/r} \left\|\sum_{i=1}^3 u_i^0 \right\|_{r/(r-1)} \|f\|_{L_r((0,T)\times\Omega)} \nonumber\\
    & \qquad \le c(T) \|f\|_{L_r((0,T)\times\Omega)}\,. \label{il18}
\end{align}
On the other hand, since
\begin{align*}
    \|\xi_{M_\varepsilon,f}\|_{L_r((0,T)\times\Omega)} & \le \big( 1 + \max\{d_1,d_2,d_3\} C_D(r) \big) \left( \int_0^T (T-t)^{r-1}\,\mathrm{d}t \right)^{1/r} \|f\|_{L_r((0,T)\times\Omega)} \\
    & \le c(T) \|f\|_{L_r((0,T)\times\Omega)}
\end{align*}
by \Cref{P4}, a similar argument gives
\begin{align*}
    & \left| \int_0^T \int_\Omega (k_1 v_{1,\varepsilon} - v_{4,\varepsilon}) \xi_{M_\varepsilon,f}\,\rd x\rd t \right| \\
    & \qquad \le \left( k_1 \|v_{1,\varepsilon}\|_{L_{r/(r-1)}((0,T)\times\Omega)} + \|v_{4,\varepsilon}\|_{L_{r/(r-1)}((0,T)\times\Omega)} \right) \|\xi_{M_\varepsilon,f}\|_{L_r((0,T)\times\Omega)}\\
    & \qquad \le c(T) (1+k_1) \left( \|v_{1,\varepsilon}\|_{L_{r/(r-1)}((0,T)\times\Omega)} + \|v_{4,\varepsilon}\|_{L_{r/(r-1)}((0,T)\times\Omega)} \right) \|f\|_{L_r((0,T)\times\Omega)}\,.
\end{align*}
Now, the choice of $r$ entails that $r\ge 2n/(n+2)$, so that $H^1(\Omega)$ embeds continuously $L_{r/(r-1)}(\Omega)$. We then infer from the Gagliardo-Nirenberg inequality and~\eqref{il11} that, for $t\in [0,T]$,
\begin{align*}
    \|v_{1,\varepsilon}(t)\|_{r/(r-1)} & \le C(r) \|v_{1,\varepsilon}(t)\|_{H^1}^{n(2-r)/(2r)} \|v_{1,\varepsilon}(t)\|_2^{[(n+2)r-2n]/(2r)} \\
    & \le c(T) \left( 1 + \|\nabla v_{1,\varepsilon}(t)\|_{2}^{n(2-r)/(2r)} \right)\,.
\end{align*}
Therefore,
\begin{equation}
    \|v_{1,\varepsilon}\|_{L_{r/(r-1)}((0,T)\times\Omega)}^{r/(r-1)} \le c(T) \left( 1 + \|\nabla v_{1,\varepsilon}(t)\|_{2}^{n(2-r)/(2r-2)} \right) \le c(T) \label{il16}
\end{equation}
by~\eqref{il12}, since $n(2-r)/(2r-2)\le 2$ due to~\eqref{r}. Owing to the contraction properties of the heat semigroup in $L_{r/(r-1)}(\Omega)$, we readily infer from~\eqref{v4e} and~\eqref{il16} that, for $t\in [0,T]$, 
\begin{align*}
    \|v_{4,\varepsilon}(t)\|_{r/(r-1)} &\le \|v_{4,\varepsilon}(0)\|_{r/(r-1)} + k_1 t^r \|v_{1,\varepsilon}\|_{L_{r/(r-1)}((0,T)\times\Omega)} \\
    & \le \|a\|_{r/(r-1)} + c(T)\le c(T)\,, 
\end{align*}
recalling that $r>p_0/(p_0-1)$. Gathering the above estimates leads us to
\begin{equation*}
    \|v_{1,\varepsilon}\|_{L_{r/(r-1)}((0,T)\times\Omega)} + \|v_{4,\varepsilon}\|_{L_{r/(r-1)}((0,T)\times\Omega)} \le c(T)
\end{equation*}
and we conclude that
\begin{equation}
   \left| \int_0^T \int_\Omega (k_1 v_{1,\varepsilon} - v_{4,\varepsilon}) \xi_{M_\varepsilon,f}\,\rd x\rd t \right| \le c(T) \|f\|_{L_r((0,T)\times\Omega)}\,. \label{il17}
\end{equation}
It now follows from~\eqref{il15}, \eqref{il18} and~\eqref{il17} that
\begin{equation*}
    \left| \int_0^T \int_\Omega f \Sigma_\varepsilon\,\rd x\rd t \right| \le c(T) \|f\|_{L_r((0,T)\times\Omega)}\,,
\end{equation*}
and a duality argument implies that 
\begin{equation*}
    \|\Sigma_\varepsilon\|_{L_{r/(r-1)}((0,T)\times\Omega)} \le c(T)\,,
\end{equation*}
from which~\eqref{il13} readily follows with $q=r/(r-1)$ due to $\Sigma_\varepsilon\ge v_{i,\varepsilon}\ge 0$ for $i\in \{1,2,3\}$.

We finally infer from~\eqref{il12}, \eqref{il13} and H\"older's inequality that
\begin{align*}
    & \int_0^T \int_\Omega \left| v_{1,\varepsilon} v_{2,\varepsilon}^2 - \varepsilon v_{2,\varepsilon}^3 \right|^{2q/(q+2)}\,\rd x\rd t \\
    & \qquad =  \int_0^T \int_\Omega \left| v_{2,\varepsilon} \big( v_{1,\varepsilon} - \varepsilon v_{2,\varepsilon} \big) \right|^{2q/(q+2)} |v_{2,\varepsilon}|^{2q/(q+2)}\,\rd x\rd t \\
    & \qquad \le \left( \int_0^T \int_\Omega \left| v_{2,\varepsilon} \big( v_{1,\varepsilon} - \varepsilon v_{2,\varepsilon} \big) \right|^2\,\rd x\rd t \right)^{q/(q+2)} \left( \int_0^T \int_\Omega |v_{2,\varepsilon}|^q\,\rd x\rd t \right)^{2/(q+2)} \\
    & \qquad \le c_3(T)\,,
\end{align*}
which is~\eqref{il14}.
\end{proof}

We are now in a position to perform the proof of \Cref{T3}.

\begin{proof}[Proof of \Cref{T3}]
Since $\min\{2,q\} > 2q/(q+2)>1$, we infer from~\eqref{il11} and~\eqref{il14} that the right-hand sides of~\eqref{v1e}, \eqref{v2e} and~\eqref{v3e} are bounded in $L_{2q/(q+2)}((0,T)\times\Omega)$ uniformly with respect to $\varepsilon\in (0,1)$. It then follows from the continuity and compactness properties of the heat semigroup in $L_{2q/(q+2)}(\Omega)$ and \cite{BHV1977} that $(v_{i,\varepsilon})_{\varepsilon\in (0,1)}$ is relatively compact in $C([0,T],L_{2q/(q+2)}(\Omega))$ for $i\in\{1,2,3\}$. Therefore, there are 
\begin{equation*}
    u_i\in C([0,T],L_{2q/(q+2)}(\Omega))\,, \quad 1\le i \le 3\,,
\end{equation*}
and a sequence $(\varepsilon_j)_{j\ge 1}$ in $(0,1)$ such that
\begin{equation}
    \lim_{j\to\infty} \varepsilon_j = 0\,, \quad \lim_{j\to\infty} \sup_{t\in [0,T]} \|(v_{i,\varepsilon_j}-u_i)(t)\|_{2q/(q+2)} = 0\,, \quad 1\le i \le 3\,. \label{il19}
\end{equation}
Upon extracting a further subsequence, we may also assume that
\begin{equation}
    \lim_{j\to\infty} v_{i,\varepsilon_j}(t,x) = u_i(t,x) \;\;\text{for a.e. }\; (t,x)\in (0,T)\times\Omega\,. \label{il20}
\end{equation}
An immediate consequence of~\eqref{il19} and~\eqref{il20} is that
\begin{equation*}
    \lim_{j\to\infty} \big( v_{1,\varepsilon_j} v_{2,\varepsilon_j}^2 - \varepsilon_j v_{3,\varepsilon_j}^3 \big)(t,x) = \big( u_1 u_2^2 \big)(t,x) \;\;\text{for a.e. }\; (t,x)\in (0,T)\times\Omega\,. 
\end{equation*}
Since $\big( v_{1,\varepsilon_j} v_{2,\varepsilon_j}^2 - \varepsilon_j v_{3,\varepsilon_j}^3 \big)_{j\ge 1}$ is weakly compact in $L_1((0,T)\times\Omega)$, Vitali's theorem, see \cite[Theorem~2.24]{FoLe2007} for instance, implies that 
\begin{equation*}
    \lim_{j\to\infty} \big\| v_{1,\varepsilon_j} v_{2,\varepsilon_j}^2 - \varepsilon_j v_{3,\varepsilon_j}^3 - u_1 u_2^2 \big\|_{L_1((0,T)\times\Omega)} = 0\,. 
\end{equation*}
Combining the above convergence with the bound~\eqref{il14}, we conclude that
\begin{equation}
    \lim_{j\to\infty} \big\| v_{1,\varepsilon_j} v_{2,\varepsilon_j}^2 - \varepsilon_j v_{3,\varepsilon_j}^3 - u_1 u_2^2 \big\|_{L_p((0,T)\times\Omega)} = 0\,, \qquad p\in \left[1,\frac{2q}{q+2} \right)\,. \label{il21}
\end{equation}
Another straightforward consequence of~\eqref{v3e} and~\eqref{il19} is that $u_3$ is the unique mild solution of~\eqref{heq3} in $L_{2q/(q+2)}(\Omega)$ on $(0,T)$.

We are left with identifying the limiting behavior of $(v_{4,\varepsilon})_{\varepsilon\in (0,1)}$ as $\varepsilon\to 0$. To this end, we deduce from~\eqref{v4e}, \eqref{il11} and Young's inequality that
\begin{align*}
    \frac{1}{2} \frac{\rd}{\rd t} \|v_{4,\varepsilon}-a\|_2^2 & = - \varepsilon \int_\Omega \nabla v_{4,\varepsilon}\cdot \nabla(v_{4,\varepsilon}-a)\,\rd x + \varepsilon \int_\Omega \big( k_1 v_{1,\varepsilon} - v_{4,\varepsilon} \big) (v_{4,\varepsilon}-a)\,\rd x \\
    & \le - \varepsilon \|\nabla v_{4,\varepsilon}\|_2^2 + \varepsilon \|\nabla v_{4,\varepsilon}\|_2 \|\nabla a\|_2 \\
    & \quad + \varepsilon \left( k_1 \| v_{1,\varepsilon}\|_2 + \|v_{4,\varepsilon}\|_2 \right) \left( \|v_{4,\varepsilon}\|_2 + \|a\|_2 \right) \\
    & \le \frac{\varepsilon}{2} \|\nabla a\|_2^2 + \varepsilon c(T) \le \varepsilon c(T)\,.
\end{align*}
Consequently, for $t\in [0,T]$,
\begin{equation*}
    \|v_{4,\varepsilon}(t)-a\|_2^2 \le \varepsilon c(T) 
\end{equation*}
and we have shown that
\begin{equation}
    \lim_{\varepsilon\to 0} \sup_{t\in [0,T]} \|v_{4,\varepsilon}(t)-a\|_2 = 0\,. \label{il22}
\end{equation}

Owing to~\eqref{il13}, \eqref{il19}, \eqref{il21} and~\eqref{il22}, we may let $j\to\infty$ in~\eqref{v1e} and~\eqref{v2e} with $\varepsilon=\varepsilon_j$ and deduce that $(u_1,u_2)$ is a mild solution to~\eqref{GS} in $L_p(\Omega,\mathbb{R}^2)$ on $(0,T)$ for any $p\in [1,2q/(q+2))$. Now, the comparison principle applied to~\eqref{gs1} ensures that
\begin{equation}
    \|u_1(t)\|_\infty \le M_1 := \max\left\{ \|u_1^0\|_\infty , \frac{\|a\|_\infty}{k_1} \right\}\,, \qquad t\in [0,T]\,. \label{il23}
\end{equation}
We next argue as in the proof of \cite[Theorem~1]{HMP1987} to improve the regularity of $u_2$. More precisely, for $\theta\in C_c^\infty((0,T)\times\Omega)$, the backward linear heat equation
\begin{equation}
\begin{split}
    \partial_t \phi & = - d_2 \Delta\phi + \phi - \theta \;\;\text{ in }\;\; (0,T)\times\Omega\,, \\
    \partial_\nu \phi & = 0 \;\;\text{ on }\;\; (0,T)\times\partial\Omega\,, \\
    \phi(T) & = 0 \;\;\text{ in }\;\; \Omega\,,
\end{split}\label{du01}
\end{equation}
has a unique classical solution $\phi$ that satisfies the following properties: for any $q\in (1,\infty)$, there is $C(q)>0$ depending only on $n$, $\Omega$, and $q$ such that
\begin{equation}
\begin{split}
    \|\phi\|_{L_q((0,T)\times\Omega)} + \|\Delta\phi\|_{L_q((0,T)\times\Omega)} & \le C(q) \|\theta\|_{L_q((0,T)\times\Omega)}\,, \\
    \|\phi(0)\|_q & \le C(q) T^{(q-1)/q} \|\theta\|_{L_q((0,T)\times\Omega)}\,,
    \end{split}\label{du00}
\end{equation}
see \cite[Lemmas~2--3]{HMP1987} and \cite{Lam1987}. Indeed, since $d_2 \Delta_N - 1$ generates an analytic semigroup of contractions in $L_2(\Omega)$, which is a contraction in $L_q(\Omega)$ for any $q\in [1,\infty]$, it follows from \cite[Th\'eor\`eme~1]{Lam1987} that, given $q\in (1,\infty)$, there is a positive constant $C_L(q)$ depending only on $n$, $\Omega$, and $d_2$ such that
\begin{equation}
    \|\partial_t \phi\|_{L_q((0,T)\times\Omega)} + \|\Delta\phi\|_{L_q((0,T)\times\Omega)} \le C_L(q) \|\theta\|_{L_q((0,T)\times\Omega)}\,. \label{du03}
\end{equation}
A first consequence of~\eqref{du01} and~\eqref{du03} is that
\begin{align*}
    \|\phi\|_{L_q((0,T)\times\Omega)} & \le \|\partial_t \phi\|_{L_q((0,T)\times\Omega)} + d_2 \|\Delta\phi\|_{L_q((0,T)\times\Omega)} + \|\theta\|_{L_q((0,T)\times\Omega)} \\
    & \le (2+d_2) C_L(q) \|\theta\|_{L_q((0,T)\times\Omega)}\,,
\end{align*}
which proves the first inequality stated in~\eqref{du00}. We next use again~\eqref{du03} and H\"older's inequality to deduce that
\begin{align*}
    \|\phi(0)\|_q & \le \int_0^T \|\partial_t\phi(t)\|_q\,\mathrm{d}t \le T^{(q-1)/q} \|\partial_t \phi\|_{L_q((0,T)\times\Omega)} \\
    & \le C_L(q) T^{(q-1)/q} \|\theta\|_{L_q((0,T)\times\Omega)}\,,
\end{align*}
and the proof of~\eqref{du00} is complete. 

We now infer from~\eqref{v1e}, \eqref{v2e}, \eqref{Vebc)} and~\eqref{du01} that
\begin{align*}
    - \int_\Omega \phi(0) \big( u_1^0 + u_2^0 \big)\,\mathrm{d}x & = \int_0^T \int_\Omega \partial_t\left(\phi\big( v_{1,\varepsilon} + v_{2,\varepsilon}\big) \right)\,\mathrm{d}x\mathrm{d}t \\
    & = \int_0^T \int_\Omega \phi \left[ d_1 \Delta v_{1,\varepsilon} - k_1 v_{1,\varepsilon} + v_{4,\varepsilon} + d_2 \Delta v_{2,\varepsilon} - v_{2,\varepsilon} + \varepsilon v_{3,\varepsilon} \right]\,\mathrm{d}x\mathrm{d}t \\
    & \quad + \int_0^T \int_\Omega \big(v_{1,\varepsilon} + v_{2,\varepsilon} \big) \left[ - d_2 \Delta\phi + \phi - \theta \right]\,\mathrm{d}x\mathrm{d}t \\
    & = (d_1-d_2) \int_0^T \int_\Omega v_{1,\varepsilon} \Delta\phi\,\mathrm{d}x\mathrm{d}t - \int_0^T \int_\Omega \theta \big( v_{1,\varepsilon} + v_{2,\varepsilon}\big)\,\mathrm{d}x\mathrm{d}t \\
    & \quad + \int_0^T \int_\Omega \phi \big[ (1-k_1) v_{1,\varepsilon} + v_{4,\varepsilon} + \varepsilon v_{3,\varepsilon} \big]\,\mathrm{d}x\mathrm{d}t\,. 
\end{align*}
Taking $\varepsilon=\varepsilon_j$ in the above identity and using the regularity of $\phi$ and $\theta$, we may take the limit $j\to\infty$ and infer from~\eqref{il19} and~\eqref{il22} that
\begin{align*}
   - \int_\Omega \phi(0) \big( u_1^0 + u_2^0 \big)\,\mathrm{d}x & = (d_1-d_2) \int_0^T \int_\Omega u_{1} \Delta\phi\,\mathrm{d}x\mathrm{d}t - \int_0^T \int_\Omega \theta \big( u_{1} + u_{2}\big)\,\mathrm{d}x\mathrm{d}t \\
    & \quad + \int_0^T \int_\Omega \phi \big[ (1-k_1) u_{1} + a \big]\,\mathrm{d}x\mathrm{d}t\,. 
\end{align*}
Therefore, owing to~\eqref{il23},
\begin{align*}
    \left| \int_0^T \int_\Omega \theta \big( u_{1} + u_{2}\big)\,\mathrm{d}x\mathrm{d}t \right| & \le \big\| u_1^0 + u_1^0\big\|_\infty \|\phi(0)\|_1 + |d_1-d_2| M_1 \|\Delta\phi\|_{L_1((0,T)\times\Omega)} \\
    & \quad + \left[ (1+k_1) M_1 + \|a\|_\infty \right] \|\phi\|_{L_1((0,T)\times\Omega)}\,. 
\end{align*}
It then readily follows from H\"older's inequality and~\eqref{du00} that, for $q\in (1,\infty)$, 
\begin{align*}
   \left| \int_0^T \int_\Omega \theta \big( u_{1} + u_{2}\big)\,\mathrm{d}x\mathrm{d}t \right| & \le c (T|\Omega|)^{(q-1)/q} \left( \|\phi(0)\|_q + \|\Delta\phi\|_{L_q((0,T)\times\Omega)} + \|\phi\|_{L_q((0,T)\times\Omega)} \right) \\
    & \le c(T) \|\theta\|_{L_q((0,T)\times\Omega)}\,,
\end{align*}
and a duality argument entails that
\begin{equation}
   \big\| u_{1} + u_{2}\big\|_{L_{q/(q-1)}((0,T)\times\Omega)} \le c(T)\,. \label{du02} 
\end{equation}
Recalling~\eqref{il23}, we conclude that $u_2\in L_{q/(q-1)}((0,T)\times\Omega)$ for any $q\in (1,\infty)$ and a classical bootstrap argument implies that $(u_1,u_2)$ is actually the unique classical solution to~\eqref{GS} on $(0,T)$, from which the convergence of the whole family $(\bv_\varepsilon)_{\varepsilon\in (0,1)}$ follows.  Also, as $T$ is arbitrary, a diagonal process guarantees the convergence on any time interval~$(0,T)$.
\end{proof}

\section*{Acknowledgements}
The work of PhL is partially funded by the Chinese Academy of Sciences President's International Fellowship Initiative Grant No.~2025PVA0101.
Part of this work was done while PhL enjoyed the hospitality of the Innovation Academy for Precision Measurement Science and Technology, Chinese Academy of Sciences, Wuhan.

\bibliographystyle{siam}
\bibliography{LitGrayScott}

\end{document}